\documentclass[11pt]{amsart}

\usepackage{amsmath,amssymb}
\usepackage{graphicx}
\usepackage{fullpage}
\usepackage{enumerate}
\usepackage{enumitem}
\usepackage{pgf,tikz}
\usetikzlibrary{patterns}
\usetikzlibrary{arrows}
\usetikzlibrary{positioning}
\usepackage{thmtools,thm-restate}

\def\COMMENT#1{}
\let\COMMENT=\footnote
\newtheorem{theorem}{Theorem}[section]

\newtheorem{lemma}[theorem]{Lemma}
\newtheorem{problem}[theorem]{Problem}
\newtheorem{claim}[theorem]{Claim}
\newtheorem{fact}[theorem]{Fact}
\newtheorem{question}[theorem]{Question}

\newtheorem{proposition}[theorem]{Proposition}

\newtheorem{exex}{Extremal Example}
\def\eps{\varepsilon}

\theoremstyle{definition}
\newtheorem{define}[theorem]{Definition}

\DeclareUnicodeCharacter{2212}{-}
\newcommand\blfootnote[1]{%
  \begingroup
  \renewcommand\thefootnote{}\footnote{#1}%
  \addtocounter{footnote}{-1}%
  \endgroup
}

\title{Ramsey-type problems for tilings in dense graphs}

\author{J\'ozsef Balogh, Andrea Freschi and Andrew Treglown}

\thanks{JB:
    Department of Mathematics, University of Illinois at Urbana--Champaign, Urbana, IL, USA. Research supported in part by NSF grants  RTG DMS-1937241, FRG DMS-2152488, the Arnold O.~Beckman Research Award (UIUC Campus Research Board RB 24012),  the Simons Fellowship,  and the BRIDGE Seed Fund (University of Birmingham and University of Illinois at Urbana--Champaign). E-mail:  \texttt{jobal@illinois.edu}.}
\thanks{AF: HUN-REN, Alfr\' ed R\'enyi Institute of Mathematics, Budapest, Hungary. Research partially supported by ERC Advanced Grants ``GeoScape'', no. 882971 and ``ERMiD'', no. 101054936, and the BRIDGE Seed Fund (University of Birmingham and University of Illinois at Urbana--Champaign).
E-mail: {\tt freschi.andrea@renyi.hu}.}
\thanks{AT: School of Mathematics, University of Birmingham, United Kingdom. 
Research supported by EPSRC grant  EP/V048287/1, and the BRIDGE Seed Fund (University of Birmingham and University of Illinois at Urbana--Champaign).
Email: \texttt{a.c.treglown@bham.ac.uk}. 
}

\begin{document}

\begin{abstract}
Given a graph $H$, the Ramsey number $R(H)$ is the smallest positive integer $n$ such that every $2$-edge-colouring of $K_n$ yields a monochromatic copy of $H$.
We write $mH$ to denote the union of $m$ vertex-disjoint copies of $H$. 
The members of the family $\{mH:m\ge1\}$ are also known as $H$-tilings.
A well-known result of Burr, Erd\H os and Spencer states that $R(mK_3)=5m$ for every $m\ge2$.
\linebreak
On the other hand, Moon proved that every $2$-edge-colouring of $K_{3m+2}$ yields a $K_3$-tiling consisting of $m$ monochromatic copies of $K_3$, for every $m\ge2$. Crucially, in Moon's result, distinct copies of $K_3$ might receive different colours.

In this paper, we investigate the analogous questions where the complete host graph is replaced by a graph of large minimum degree.
We determine the (asymptotic) minimum degree threshold for forcing a~\hbox{$K_3$-tiling} covering a prescribed proportion of the vertices in a~\hbox{$2$-edge-coloured} graph such that every copy of~$K_3$ in the tiling is monochromatic.
We also determine the largest size of a monochromatic $K_3$-tiling one can guarantee in any $2$-edge-coloured graph of large minimum degree. 
These results therefore provide generalisations of the theorems of Moon and Burr--Erd\H os--Spencer  to the
setting of dense graphs.

It is also natural to consider generalisations of these problems to $r$-edge-colourings (for $r \geq 2$) and for $H$-tilings (for arbitrary graphs $H$).
We prove some results in this direction and propose several open questions.
\end{abstract}

\maketitle

\blfootnote{The main results of this paper were first announced in the conference abstract~\cite{conf}.}

\section{Introduction}

Ramsey theory is a central research topic in combinatorics.
Ramsey's original theorem~\cite{Ramsey} asserts that for every~$r\in\mathbb N$ and every graph~$H$, there exists an~$n\in\mathbb N$ such that every \hbox{$r$-edge-colouring} of the complete graph~$K_n$  yields a monochromatic copy of~$H$.
We write $R_r(H)$ to denote the smallest~$n$ for which the above holds, and set $R(H):=R_2(H)$.

In general, determining $R(H)$ is a very difficult problem and there are relatively few graphs $H$ for which the exact value of $R(H)$ is known.
So-called tilings provide 
an interesting class of graphs whose Ramsey behaviour is quite well-understood.
For a fixed graph $H$, an {\it $H$-tiling} is a collection of \hbox{vertex-disjoint} copies of $H$.
For $m\in\mathbb N$, we write $mH$ to denote an $H$-tiling consisting of $m$ copies of $H$.
Erd\H os~\cite[Problem 9]{erdosproblem} raised the question of determining $R(mK_{\ell})$ for $\ell \geq 3$.
The following  result of Burr, Erd\H os and Spencer~\cite{BurrES} resolves this question for $\ell =3$.

\begin{theorem}[Burr, Erd\H os and Spencer~\cite{BurrES}]\label{theorem:BES}
For every integer $m\ge2$, we have $R(mK_3)=5m$.    
\end{theorem}

More generally, Burr, Erd\H os and Spencer~\cite{BurrES} proved that, for a fixed graph~$H$ without isolated vertices, there exist constants $c$ and $m_0$ such that $R(mH)=(2|H|-\alpha(H))m+c$ provided $m\ge m_0$, where~$\alpha(H)$ is the independence number of~$H$.
Burr~\cite{Burr}, and subsequently Buci\'c and Sudakov~\cite{BucicS}, provided methods for computing~$c$ exactly.
Buci\'c and Sudakov~\cite{BucicS} also obtained the current best bounds for~$m_0$. In the case of $K_\ell$-tilings, their work states that there is a constant $C>0$ such 
 that $R(mK_\ell)=(2\ell-1)m+R(\ell-1)-2$ provided $m\ge 2^{C\ell}$. Moreover, the bound on $m$ is essentially tight; see~\cite{BucicS}.

Although not a Ramsey-type question in the classical sense, it is also natural to ask how large a complete $r$-edge-coloured graph needs to be to ensure there exists an $H$-tiling of a given size such that every copy of $H$ is monochromatic.
Crucially, in this setting, different copies of $H$ in the tiling are allowed to receive different colours.
This problem was  studied prior to the work of Burr, Erd\H os and Spencer~\cite{BurrES}. Indeed,  
the following result of Moon~\cite{Moon}  settles the $H=K_3$ case of this problem.

\begin{theorem}[Moon~\cite{Moon}]\label{theorem:Moon}
For every integer $m\ge2$, every $2$-edge-colouring of $K_{3m+2}$ yields a $K_3$-tiling consisting of $m$ monochromatic copies of $K_3$.
Furthermore, the term $3m+2$ cannot be replaced by a smaller integer.
\end{theorem}
 
Burr, Erd\H os and Spencer~\cite{BurrES} proved an analogue of Theorem~\ref{theorem:Moon} for larger cliques: any $2$-edge-colouring of $K_{\ell m+R(\ell,\ell-1)-1}$ yields $m$ vertex-disjoint monochromatic copies of $K_\ell$ provided $m$ is sufficiently large with respect to $\ell$; again the $\ell m+R(\ell,\ell-1)-1$ term is best possible.
Here, we write $R(r,b)$ to denote the smallest $n$ such that any red/blue edge-colouring of $K_n$ yields  a red $K_r$ or a blue $K_b$.

We remark that there are many other Ramsey-type results concerning finding many vertex-disjoint monochromatic copies of graphs from a given graph family. Paths and cycles have received particular attention.
Erd\H{o}s and Gy\'arf\'as~\cite{ErdosG} proved that the vertex set of a \hbox{$2$-edge-coloured} complete graph $K_n$ can be partitioned into at most~$2\sqrt{n}$ \hbox{vertex-disjoint} monochromatic paths, all of the same colour, and conjectured the $2\sqrt{n}$ term can be replaced by $\sqrt{n}$.
This was proved recently by Pokrovskiy, Versteegen and Williams~\cite{Pokrovskiy} for sufficiently large $n$.
Gerencs\'er and Gy\'arf\'as~\cite{GerencserG} observed that the vertex set of any \hbox{$2$-edge-coloured} complete graph can be partitioned into two \hbox{vertex-disjoint} monochromatic paths of different colours.
In the 1970s, Lehel conjectured that the vertex set of any \hbox{$2$-edge-coloured} complete graph~$K_n$ can be partitioned into two monochromatic cycles of different colours (see, e.g.,~\cite{Ayel}).
Lehel's conjecture was proved for large~$n$ by {\L}uczak, R\"odl and Szemer\'edi~\cite{LuczakRS}.
The bound on~$n$ was later improved by Allen~\cite{Allen}, and finally Bessy and Thomass\'e~\cite{BessyT} provided a full resolution of Lehel's conjecture.

\medskip

Schelp~\cite{Schelp} (see also~\cite{li}) proposed the study of Ramsey-type questions where the host graph, rather than being complete, can be any graph satisfying a given minimum degree condition.
Several results have been proved in this direction.
For example, Schelp~\cite{Schelp} conjectured that any $2$-edge-coloured $n$-vertex graph $G$ with minimum degree $\delta(G)\ge3n/4$ contains a monochromatic path of length at least $2n/3$.
This conjecture was verified asymptotically by Gy\'arf\'as and S\'ark\"ozy~\cite{GyarfasS}. 
Balogh, Bar\'at, Gerbner, Gy\'arf\'as and S\'ark\"ozy~\cite{BaloghBGGS} conjectured that the conclusion of Lehel's conjecture still holds for any \hbox{$2$-edge-coloured} $n$-vertex graph $G$ with minimum degree~$\delta(G)\ge3n/4$, and proved an asymptotic version of this statement.
A stronger asymptotic result was proved by DeBiasio and Nelsen~\cite{DeBiasioN}
and an exact statement (for~$n$ sufficiently large) was proved by Letzter~\cite{Letzter}.

Motivated by this line of research, in this paper we consider the natural generalisations of the aforementioned classical Ramsey-type results about tilings to the dense setting.
The works of Burr--Erd\H os--Spencer and Moon suggest the following two problems.
In the former, one is interested in finding a large monochromatic $H$-tiling in a graph with given minimum degree. 
The latter problem is the same except we only insist that individual copies of $H$ in the $H$-tiling are monochromatic, but different copies of $H$ may receive different colours.

\begin{problem}\label{problem:alaBES}
Let~$H$ be a fixed graph and~$n,r,\delta\in\mathbb N$.
Determine the largest~$m\in\mathbb N$ such that any \hbox{$r$-edge-coloured} $n$-vertex graph~$G$ with minimum degree~$\delta(G)\ge\delta$ contains a monochromatic copy of~$mH$.
\end{problem}

\begin{problem}\label{problem:alaMoon}
Let~$H$ be a fixed graph and~$n,r,\delta\in\mathbb N$.
Determine the largest~$m\in\mathbb N$ such that any \hbox{$r$-edge-coloured} $n$-vertex graph~$G$ with minimum degree~$\delta(G)\ge\delta$ contains an \hbox{$H$-tiling} consisting of~$m$ monochromatic copies of~$H$.
\end{problem}

Note that Theorems~\ref{theorem:BES} and~\ref{theorem:Moon} provide a full resolution of the case $H=K_3$, $r=2$ and $\delta=n-1$ of Problems~\ref{problem:alaBES} and~\ref{problem:alaMoon} respectively. Indeed, for this case of Problem~\ref{problem:alaBES} and when $n \geq 6$,
we have that $m$ is the integer such that $n=5m+k$ for some $0\leq k \leq 4$. Further,  for this case of Problem~\ref{problem:alaMoon}
and when $n \geq 6$, we have that  $m$ is the integer such that $n=3m+k$ for some $2\leq k \leq 4$.

The  $r=1$ case of both Problems~\ref{problem:alaBES} and~\ref{problem:alaMoon} is equivalent to determining the largest $H$-tiling one can guarantee in any $n$-vertex graph $G$ with $\delta(G)\ge\delta$.
By itself, this case of the problem has received considerable attention, and has motivated a fruitful line of research. 
An \hbox{$H$-tiling} in a graph~$G$ is {\it perfect} if it contains all the vertices of~$G$.
Corr\'adi and Hajnal~\cite{CorradiH} determined the minimum degree threshold that guarantees the existence of a perfect \hbox{$K_3$-tiling}.
This result was further generalised to perfect \hbox{$K_t$-tilings} (for every~$t\in\mathbb N$) by Hajnal and Szemer\'edi~\cite{HajnalS} and to perfect \hbox{$H$-tilings} (for every fixed graph~$H$) by K\"uhn and Osthus~\cite{KuhnO}.
Combining the Hajnal--Szemer\'edi theorem with an elementary interpolation argument, one can easily determine the minimum degree threshold to force a $K_t$-tiling covering a fixed proportion of the vertices (see, e.g., Theorem~\ref{thm:HajnalSzm2} in this paper). 
The same elementary strategy fails for \hbox{$H$-tilings} where~$H$ is an arbitrary fixed graph~$H$.
Koml\'os~\cite{Komlos} determined (asymptotically) the minimum degree threshold that guarantees the existence of an \hbox{$H$-tiling} covering a fixed proportion of the vertices of the host graph, for any fixed graph~$H$, provided the fixed proportion is less than~$1$.
Therefore, the $r=1$ case of Problems~\ref{problem:alaBES} and~\ref{problem:alaMoon} is (asymptotically) fully understood.

The $H=K_2$ case of both Problems~\ref{problem:alaBES} and~\ref{problem:alaMoon} has also been resolved.
Indeed, the case $H=K_2$ of Problem~\ref{problem:alaMoon} is equivalent to determining the largest $K_2$-tiling in a graph with given minimum degree, and thus it is covered by, for example, the Hajnal--Szemer\'edi theorem.
The case $H=K_2$ of Problem~\ref{problem:alaBES} has a more interesting history. Given graphs $H_1,\dots, H_r$,
we write $R_r(H_1,\dots,H_r)$ to denote the smallest integer $n$ such that any $r$-edge-colouring of $K_n$ using colours $c_1,\dots,c_r$ yields a monochromatic copy of $H_i$ in colour $c_i$, for some $i$.
Cockayne and Lorimer~\cite{CockayneL} proved that $R_r(mK_2)=(r+1)(m-1)+2$, resolving the case $H=K_2$, $\delta=n-1$ of Problem~\ref{problem:alaBES}.
Gy\'arf\'as and S\'ark\"ozy~\cite{GyarfasS} determined
$R(mK_2,mK_2,S_t)$ for all $t,m\in\mathbb N$, where $S_t$ is the star on $t+1$ vertices.
The connection of this purely Ramsey-type result to Problem~\ref{problem:alaBES} is that a red/blue/green edge-coloured $K_n$ which does not contain a green monochromatic copy of $S_t$ can be seen as a  red/blue edge-coloured $n$-vertex graph $G$ with $\delta(G)\ge n-t$.
Therefore, Gy\'arf\'as and S\'ark\"ozy's result resolves the case $H=K_2$, $r=2$ of Problem~\ref{problem:alaBES}.
Finally, Omidi, Raeisi and Rahimi~\cite{OmidiRR} computed $R_r(mK_2,\dots,mK_2,S_t)$ for all $r,t,m\in\mathbb N$, thus resolving the case $H=K_2$ of Problem~\ref{problem:alaBES} in full.

\subsection{Main results}
In this paper, our main focus is to study the case $H=K_3$, $r=2$ of Problems~\ref{problem:alaBES} and~\ref{problem:alaMoon}.
Observe that the case $\delta \leq 4n/5$ is uninteresting, as one cannot guarantee a single monochromatic copy of $K_3$.
Indeed, consider a balanced complete $n$-vertex $5$-partite graph~$G$ with classes $V_1,\dots,V_5$.
Clearly $\delta(G)=\lfloor 4n/5\rfloor$.
Colour all edges between $V_i$ and $V_{i+1}$ red, where the indices are taken modulo $5$.
All remaining edges are blue.
Thus,  $G$ does not contain a monochromatic copy of $K_3$.

For the case $H=K_3$, $r=2$ of Problem~\ref{problem:alaBES}, the following theorem provides an exact answer when $\delta$ is a bit larger than $4n/5$ or a bit smaller than $n-1$.

\begin{theorem}\label{theorem:alaBES}
Let $n\in\mathbb N$ and
 $G$ be a $2$-edge-coloured $n$-vertex graph.
Then $G$ contains a monochromatic copy of $mK_3$ where $m$ is equal to 

\begin{enumerate}[label=(B.\arabic*)]

    \vspace{0.2cm}

    \item\label{case:BESlarge} \makebox[3cm]{$\left\lfloor\frac{\delta(G)+1}{5}\right\rfloor$} if \quad $\frac{65n}{66} \leq \delta(G)$,

    \vspace{0.5cm}

    \item\label{case:BESsmall}  \makebox[3cm]{$\left\lceil\frac{5\delta(G)-4n}{2}\right\rceil$} if \quad $\frac{4n}{5}\le\delta(G)\le\frac{5n}{6}.$
\end{enumerate}
\vspace{0.1cm}
Furthermore, parts~\ref{case:BESlarge} and~\ref{case:BESsmall} are best possible, in the sense that the statement of the theorem does not hold if $m$ is replaced by a larger number.
\end{theorem}

 Case~\ref{case:BESlarge} of Theorem~\ref{theorem:alaBES} can be seen as a dense generalisation of the Burr--Erd\H os--Spencer result, as
 Theorem~\ref{theorem:BES} corresponds precisely to the case $n=5m$, $\delta(G)=n-1$ of Theorem~\ref{theorem:alaBES}.

Theorem~\ref{theorem:alaBES} does not cover graphs with minimum degree between $5n/6$ and $65n/66$, however, 
 we raise the following question.

\begin{question}\label{conjecture:alaBES}
Is the following true? Let $n\geq 25$ 
be an integer and  $G$ be a $2$-edge-coloured $n$-vertex graph.
Then $G$ contains a monochromatic copy of $mK_3$, where $m$ is equal to 
\begin{enumerate}[label=(C.\arabic*)]

    \vspace{0.2cm}

    \item\label{case:conjBESlarge} \makebox[3cm]{$\left\lfloor\frac{\delta(G)+1}{5}\right\rfloor$} if \quad $ \frac{15n}{17} \le \delta(G) $, 

    \vspace{0.5cm}

    \item\label{case:conjBESmedium}  \makebox[3cm]{$\left\lfloor\frac{4\delta(G)-3n+1}{3}\right\rfloor$} if \quad $\frac{6n}{7}\le\delta(G)\le\frac{15n}{17}$,

    \vspace{0.5cm}

    \item\label{case:conjBESsmall}  \makebox[3cm]{$\left\lceil\frac{5\delta(G)-4n}{2}\right\rceil$} if \quad $\frac{4n}{5}\le\delta(G)\le\frac{6n}{7}.$
\end{enumerate}    
\end{question}
In Section~\ref{sec2.2} we provide extremal examples that show the bounds on $m$ in Question~\ref{conjecture:alaBES} cannot be increased.
Note that  we put the condition $n \geq 25$ in Question~\ref{conjecture:alaBES} 
to ensure that we can separate into three cases and also to ensure that we have matching extremal examples. It may be possible that there is an affirmative answer to the question with a smaller lower bound on $n$.

\smallskip

When $H=K_3$ and $r=2$, Problem~\ref{problem:alaMoon} turns out to be much more tractable.
The following theorem provides an (asymptotic) resolution of this case.

\begin{theorem}\label{theorem:alaMoon}
Let~$n\in\mathbb N$ and  $G$ be a \hbox{$2$-edge-coloured} $n$-vertex graph.
Then there exists a \hbox{$K_3$-tiling} in~$G$ such that every copy of~$K_3$ is monochromatic and the number of copies of~$K_3$ in the tiling is at least

\begin{enumerate}[label=(M.\arabic*)]

    \item\label{case:Moonlarge}   \makebox[3.5cm]{$\left\lfloor\frac{2\delta(G)-n}{3}\right\rfloor$} if \quad $\frac{7n}{8}\leq \delta(G)$, 

    \vspace{0.5cm}
    
    \item\label{case:Moonmedium}   \makebox[3.5cm]{$\left\lfloor\frac{4\delta(G)-3n}{2}\right\rfloor-o(n)$} if \quad $\frac{5n}{6}\le\delta(G)\le\frac{7n}{8}$, 

    \vspace{0.5cm}

    \item\label{case:Moonsmall}  \makebox[3.5cm]{$5\delta(G)-4n$} if \quad $\frac{4n}{5}\le \delta(G)\le \frac{5n}{6}$.

    \vspace{0.5cm}

\end{enumerate} 
\noindent Furthermore, parts~\ref{case:Moonlarge} and~\ref{case:Moonsmall} are best possible and part~\ref{case:Moonmedium} is best possible up to the~$o(n)$ term.
\end{theorem}
Note that for $n \geq 8$, \ref{case:Moonlarge} deals with the case when the host graph $G$ is complete, and so generalises Theorem~\ref{theorem:Moon}. On the other hand, \ref{case:Moonsmall} for $n=5$ reiterates that there exists a $2$-edge-coloured $K_5$ without a monochromatic $K_3$; the $n=6$ case reiterates that every $2$-edge-coloured $K_6$ contains a monochromatic $K_3$.

\subsection{Organisation of the paper and notation}

In the next section, we present the extremal examples showing the sharpness of Theorems~\ref{theorem:alaBES} and~\ref{theorem:alaMoon}, and the bounds in Question~\ref{conjecture:alaBES}.
The third and fourth sections cover the proofs of Theorems~\ref{theorem:alaMoon} and~\ref{theorem:alaBES}, respectively.
In the final section we discuss some further results and research directions.
We conclude this section with a list of the notation used throughout the paper. 

\smallskip

Given $n \in \mathbb N$, we set 
$[n]:=\{1,2,\dots,n\}$.
Given two sets~$A$ and~$B$, we write~$A\dot\cup B$ to denote the disjoint union of~$A$ and~$B$.
For a graph $G$, we write $|G|$ to denote the number of vertices in $G$.
A set of vertices $S\subseteq V(G)$ is {\it independent} if no edge lies in it.
A subgraph~$H$ of~$G$ is {\it spanning} if~$V(H)=V(G)$.
Given a set~$X\subseteq V(G)$, we write~$G[X]$ for the {\it induced subgraph of~$G$ on~$X$}, that is, the subgraph with vertex set~$X$ which contains all edges of~$G$ lying in~$X$. Set $G\setminus X:=G[V(G)\setminus X]$.

Given a partition~$V_1,\dots,V_k$ of~$V(G)$, we  write~$G[V_1,\dots,V_k]$ to denote the spanning subgraph of~$G$ containing all edges of~$G$ except those lying within a class~$V_i$, for any~$i\ge1$.
Given a graph~$H$ and~$k\in\mathbb N$, we write~$H(k)$ to denote the blow-up of~$H$ where every vertex is replaced by a class of~$k$ vertices.
A \emph{blow-up} of an edge-coloured graph $G$ is an edge-coloured $|G|$-partite graph with vertex classes $\{V_v:v\in V(G)\}$ such that if $xy\in E(G)$ then all edges between $V_x$ and $V_y$ are present and have the same colour as $xy$, whereas if $xy\notin E(G)$ then there is no edge between $V_x$ and $V_y$.

We say that a \hbox{$2$-edge-coloured}~$K_5$ is {\it badly coloured} if the edges of each colour form a cycle of length~$5$; so
 a badly coloured~$K_5$ does not contain a monochromatic copy of~$K_3$.

\section{Extremal examples for Theorems~\ref{theorem:alaBES}, Theorem~\ref{theorem:alaMoon} and Question~\ref{conjecture:alaBES}}

In this section we present extremal examples that show the bounds on the size of the $K_3$-tilings in Theorems~\ref{theorem:alaBES} and~\ref{theorem:alaMoon}, and Question~\ref{conjecture:alaBES} cannot be increased. 

\subsection{Extremal examples for Theorem~\ref{theorem:alaMoon}}

The following construction shows the sharpness of Theorem~\ref{theorem:alaMoon} for all its cases.

\begin{exex}\label{exex:Moon}
For every~$n,\delta\in\mathbb N$ such that~$4n/5\le\delta\le n-1$, we write~${\text{EX}}_{\triangle}(n,\delta)$ to denote the following red/blue \hbox{edge-coloured} graph.
We have~$V({\text{EX}}_{\triangle}(n,\delta))=V_0\dot\cup V_1\dot\cup\cdots\dot\cup V_5$ where~$|V_i|=n-\delta\ge1$ for every~$i\ge1$ and~$|V_0|=5\delta-4n\ge0$.
The sets~$V_1,\dots,V_5$ are independent; all other pairs of vertices form an edge.
The subgraph~${\text{EX}}_{\triangle}(n,\delta)[V_0\cup V_1,V_2,V_3,V_4,V_5]$ is a \hbox{blow-up} of a badly coloured~$K_5$.
The edges lying in $V_0$ and the edges incident to both $V_0$ and $V_1$ are red.
See Figure~\ref{fig:EX1} for a representation of~${\text{EX}}_{\triangle}(n,\delta)$. 
\end{exex}

\begin{figure}[h!]
\centering
\begin{tikzpicture}[scale=0.76]

\filldraw[color=red!50,fill=red!50,very thick] (18:3.3)--(90:3.3)--(162:3.3)--(234:3.3)--(306:3.3)--(18:3.3);

\filldraw[color=red!50,fill=white,very thick] (18:2.7)--(90:2.7)--(162:2.7)--(234:2.7)--(306:2.7)--(18:2.7);

\filldraw[color=blue!50,fill=blue!50,very thick] (18:3.6)--(162:3.6)--(162:2.4)--(18:2.4)--(18:3.6);

\filldraw[color=blue!50,fill=blue!50,very thick] (18:3.6)--(234:3.6)--(234:2.4)--(18:2.4)--(18:3.6);

\filldraw[color=blue!50,fill=blue!50,very thick] (90:3.6)--(234:3.6)--(234:2.4)--(90:2.4)--(90:3.6);

\filldraw[color=blue!50,fill=blue!50,very thick] (90:3.6)--(306:3.6)--(306:2.4)--(90:2.4)--(90:3.6);

\filldraw[color=blue!50,fill=blue!50,very thick] (162:3.6)--(306:3.6)--(306:2.4)--(162:2.4)--(162:3.6);

\filldraw[color=black,fill=white,very thick] (18:3) circle (1cm);

\filldraw[color=black,fill=white,very thick] (90:3) circle (1.6cm);

\filldraw[color=black,fill=white,very thick] (162:3) circle (1cm);

\filldraw[color=black,fill=white,very thick] (234:3) circle (1cm);

\filldraw[color=black,fill=white,very thick] (306:3) circle (1cm);

\filldraw[color=red!50,fill=red!50,very thick] (100:4)--(68:3.3)--(68:2.7)--(100:2.25)--(100:4);

\filldraw[color=black,fill=white,very thick] (100:3) circle (1cm);

\filldraw[color=black,fill=red!50,very thick] (68:3) circle (0.4cm);

\node at (18:3) {$V_5$};
\node at (162:3) {$V_2$};
\node at (234:3) {$V_3$};
\node at (306:3) {$V_4$};

\node at (100:3) {$V_1$};
\node at (68:3) {$V_0$};
    
\end{tikzpicture}
\caption{A representation of the graph ${\text{EX}}_{\triangle}(n,\delta)$.
Note that removing all edges inside $V_0\cup V_1$ yields an unbalanced blow-up of $K_5$ with no monochromatic triangle.}
\label{fig:EX1}
\end{figure}

In the next lemma, we determine an upper bound for the largest $K_3$-tiling in $\text{EX}_{\triangle}(n,\delta)$ consisting of monochromatic copies of $K_3$.

\begin{lemma}\label{lemma:exMoon}
Let~$n,\delta\in\mathbb N$ such that~$4n/5\le\delta\le n-1$.
Then~${\text{EX}}_{\triangle}(n,\delta)$ is an $n$-vertex graph with minimum degree~$\delta({\text{EX}}_{\triangle}(n,\delta))=\delta$.
Furthermore, for any collection~$\mathcal F$ of \hbox{vertex-disjoint} monochromatic copies of~$K_3$ in~${\text{EX}}_{\triangle}(n,\delta)$, we have

$$|\mathcal F|\le\min\left\{5\delta-4n\;,\;\frac{4\delta-3n}{2}\;,\;\frac{2\delta-n}{3}\right\}.$$

\end{lemma}

\begin{proof}
We have~$|\text{EX}_{\triangle}(n,\delta)|=|V_0|+\ldots+|V_5|=(5\delta-4n)+5(n-\delta)=n$.
Every vertex in~$V_0$ has degree~$n-1$, all other vertices have degree~$\delta$.
In particular,~$\delta(\text{EX}_{\triangle}(n,\delta))=\delta$.
Let~$\mathcal F$ be a collection of \hbox{vertex-disjoint} monochromatic copies of~$K_3$ in~$\text{EX}_{\triangle}(n,\delta)$.

As~$\text{EX}_{\triangle}(n,\delta)[V_0\cup V_1,V_2,V_3,V_4,V_5]$ is a blow-up of a badly coloured~$K_5$, it does not contain a monochromatic copy of~$K_3$.
It follows that every monochromatic copy of~$K_3$ contains an edge lying in~$V_0\cup V_1$, and thus it must be red.
In particular, every monochromatic copy of~$K_3$ (i) has at least one vertex in~$V_0$ (since $V_1$ is independent) and (ii) at least two vertices in~$V_0\cup V_1$.

Recall that the blue edges of a badly coloured~$K_5$ form a cycle of length~$5$ and so each vertex of~$K_5$ is incident to two blue edges.
As~$\text{EX}_{\triangle}(n,\delta)[V_0\cup V_1,V_2,V_3,V_4,V_5]$ is a blow-up of a badly coloured~$K_5$, we may assume without loss of generality that all edges between~$V_0\cup V_1$ and~$V_3\cup V_4$ are blue.
It follows that (iii) no red monochromatic copy of~$K_3$ intersects~$V_3\cup V_4$.
Property (i) implies~$|\mathcal F|\le |V_0|=5\delta-4n$.
Property (ii) implies~$|\mathcal F|\le |V_0\cup V_1|/2=(4\delta-3n)/2$.
Property (iii) implies~$|\mathcal F|\le (n-|V_3\cup V_4|)/3=(2\delta-n)/3$.
\end{proof}

It is easy to check that

\begin{equation}\label{eq:intervals}
\min\left\{5\delta-4n\;,\;\frac{4\delta-3n}{2}\;,\;\frac{2\delta-n}{3}\right\}
=\begin{cases}
(2\delta-n)/3 & \text{ if \quad $7n/8\le\delta$;}\\
(4\delta-3n)/2 & \text{ if \quad $5n/6\le\delta\le 7n/8$;}\\
5\delta-4n & \text{ if \quad $4n/5\le\delta\le 5n/6$.} 
\end{cases}
\end{equation}

Lemma~\ref{lemma:exMoon} and equation~\eqref{eq:intervals} imply parts~\ref{case:Moonlarge} and~\ref{case:Moonsmall} of Theorem~\ref{theorem:alaMoon} are best possible, and part~\ref{case:Moonmedium} is best possible up to the~$o(n)$ term.

The next construction is an alternative extremal example for part~\ref{case:Moonlarge} of Theorem~\ref{theorem:alaMoon}.

\begin{exex}\label{exex:MoonAlt}
For every~$n,\delta\in\mathbb N$ such that~$7n/8\le\delta\le n-1$, we write~${\text{EX}}_{\triangle}^{\text{Alt}}(n,\delta)$ to denote the following red/blue \hbox{edge-coloured} graph.
We have~$V({\text{EX}}_{\triangle}^{\text{Alt}}(n,\delta))=S_1\dot\cup S_2\dot\cup R$ where~$|S_1|=|S_2|=n-\delta\ge1$ and~$|R|=2\delta-n\ge1$.
The sets~$S_1$ and~$S_2$ are independent.
All other pairs of vertices form an edge.
All edges with one vertex in~$S_1\cup S_2$ and the other in~$R$ are blue.
All remaining edges are red.
See Figure~\ref{fig:EX2} for a representation of~${\text{EX}}_{\triangle}^{\text{Alt}}(n,\delta)$.
\end{exex}

\begin{figure}[h!]
\centering
\begin{tikzpicture}[scale=0.76]

\filldraw[color=red!50,fill=red!50,very thick] (0.3,0)--(0.3,3)--(-0.3,3)--(-0.3,0)--(0.3,0);

\filldraw[color=blue!50,fill=blue!50,very thick] (0.5,0)--(4,1.5)--(3,1.5)--(-0.5,0)--(0.5,0);

\filldraw[color=blue!50,fill=blue!50,very thick] (0.5,3)--(4,1.5)--(3,1.5)--(-0.5,3)--(0.5,3);

\filldraw[color=black,fill=white,very thick] (0,0) circle (1cm);

\filldraw[color=black,fill=white,very thick] (0,3) circle (1cm);

\filldraw[color=black,fill=red!50,very thick] (3.5,1.5) circle (2cm);

\node at (0,0) {$S_1$};
\node at (0,3) {$S_2$};
\node at (3.5,1.5) {$R$};
    
\end{tikzpicture}
\caption{A representation of the graph~${\text{EX}}_{\triangle}^{\text{Alt}}(n,\delta)$.}
\label{fig:EX2}
\end{figure}

The next lemma provides an upper bound for the largest $K_3$-tiling in $\text{EX}_{\triangle}^{\text{Alt}}(n,\delta)$ consisting of monochromatic copies of $K_3$; this bound matches exactly part~\ref{case:Moonlarge} of Theorem~\ref{theorem:alaMoon}.

\begin{lemma}\label{lemma:exMoonAlt}
Let~$n,\delta\in\mathbb N$ such that~$7n/8\le\delta\le n-1$.
Then~$\text{EX}_{\triangle}^{\text{Alt}}(n,\delta)$ is an $n$-vertex graph with minimum degree~$\delta(\text{EX}_{\triangle}^{\text{Alt}}(n,\delta))=\delta$.
Furthermore, for any collection~$\mathcal F$ of \hbox{vertex-disjoint} monochromatic copies of~$K_3$ in~$\text{EX}_{\triangle}^{\text{Alt}}(n,\delta)$, we have~$|\mathcal F|\le (2\delta-n)/3$.
\end{lemma}

\begin{proof}
We have~$|\text{EX}_{\triangle}^{\text{Alt}}(n,\delta)|=|S_1|+|S_2|+|R|=2(n-\delta)+(2\delta-n)=n$.
Every vertex in~$S_1\cup S_2$ has degree~$\delta$; all remaining vertices have degree~$n-1$.
In particular,~$\delta(\text{EX}_{\triangle}^{\text{Alt}}(n,\delta))=\delta$.

Let~$\mathcal F$ be a collection of \hbox{vertex-disjoint} monochromatic copies of~$K_3$ in~$\text{EX}_{\triangle}^{\text{Alt}}(n,\delta)$.
Observe that the blue edges of~$\text{EX}_{\triangle}^{\text{Alt}}(n,\delta)$ form a complete bipartite graph, hence there is no blue monochromatic copy of~$K_3$.
Moreover, any red monochromatic copy of~$K_3$ must lie in~$R$.
Since~$|R|=2\delta-n$, it follows that~$|\mathcal F|\le(2\delta-n)/3$.
\end{proof}

\subsection{Extremal examples for Theorem~\ref{theorem:alaBES} and Question~\ref{conjecture:alaBES}}\label{sec2.2}

We have three different constructions.
We start with the one for large degree, which proves the sharpness of part~\ref{case:conjBESlarge} of Question~\ref{conjecture:alaBES} and thus of part~\ref{case:BESlarge} of Theorem~\ref{theorem:alaBES}.

\begin{exex}\label{exex:BES1}
For every~$n,\delta\in\mathbb N$ with $5 \leq \delta\le n-1$,\footnote{We need $5\leq \delta$ here to ensure $|B|\ge 0$.} 
we write~$\text{EX}_{\blacktriangle}^1(n,\delta)$ to denote the following red/blue \hbox{edge-coloured} graph.
We have~$V(\text{EX}_{\blacktriangle}^1(n,\delta))=R\dot\cup B\dot\cup S$ where~$|S|=n-\delta\ge1$, $|R|=3\lfloor(\delta+1)/5\rfloor+2\ge 5$ and $|B|=n-|S|-|R|\ge0$.
The set~$S$ is independent and 
all other pairs of vertices form an edge.
All edges lying in $R$ as well as the edges incident to both $S$ and $B$ are red.
All edges lying in $B$ as well as the edges incident to both $R$ and $S\cup B$ are blue.
See Figure~\ref{fig:EX3} for a representation of~$\text{EX}_{\blacktriangle}^1(n,\delta)$.
\end{exex}

\begin{figure}[h!]
\centering
\begin{tikzpicture}[scale=0.76]

\filldraw[color=red!50,fill=red!50,very thick] (0,0.5)--(2,3.5)--(2,2.5)--(0,-0.5)--(0,0.5);

\filldraw[color=blue!50,fill=blue!50,very thick] (0,-0.5)--(4.5,-1.5)--(4.5,1.5)--(0,0.5)--(0,-0.5);

\filldraw[color=blue!50,fill=blue!50,very thick] (4.9,0)--(2.4,3)--(1.6,3)--(4.1,0)--(4.9,0);

\filldraw[color=black,fill=blue!50,very thick] (0,0) circle (1.6cm);

\filldraw[color=black,fill=white,very thick] (2,3) circle (1cm);

\filldraw[color=black,fill=red!50,very thick] (4.5,0) circle (2cm);

\node at (0,0) {$B$};
\node at (2,3) {$S$};
\node at (4.5,0) {$R$};
    
\end{tikzpicture}
\caption{A representation of the graph~$\text{EX}_{\blacktriangle}^1(n,\delta)$.
The ratio $|B|/|R|$ is approximately $2/3$.}
\label{fig:EX3}
\end{figure}

In the next lemma, we give an upper bound for the largest $m$ such that $\text{EX}_{\blacktriangle}^1(n,\delta)$ contains a monochromatic $mK_3$.

\begin{lemma}\label{lemma:BES1}
Let~$n,\delta\in\mathbb N$ with $5 \leq \delta\le n-1$.
Then~$\text{EX}_{\blacktriangle}^1(n,\delta)$ is an $n$-vertex graph with minimum degree~$\delta(\text{EX}_{\blacktriangle}^1(n,\delta))=\delta$.
Furthermore, for any monochromatic copy of~$mK_3$ in~$\text{EX}_{\blacktriangle}^1(n,\delta)$ we have $m\le\lfloor(\delta+1)/5\rfloor$.
\end{lemma}

\begin{proof}
We have~$|\text{EX}_{\blacktriangle}^1(n,\delta)|=|R|+|B|+|S|=n$.
Every vertex in~$S$ has degree~$\delta$, all other vertices have degree $n-1$,
implying $\delta(\text{EX}_{\blacktriangle}^1(n,\delta))=\delta$.
Also, note that
$$(3\lfloor(\delta+1)/5\rfloor+2)
+(2\lfloor(\delta+1)/5\rfloor+1)
=5\lfloor(\delta+1)/5\rfloor+3
\ge5(\delta-3)/5+3=\delta.$$
Since $|B|=n-|S|-|R|=\delta-|R|$ it follows that $|B|\le2\lfloor(\delta+1)/5\rfloor+1$.

Observe that there is no monochromatic  $K_3$ intersecting $S$, i.e.,  every monochromatic copy of $K_3$ must lie in $R\cup B$.
In particular, a red copy of $K_3$ must lie completely in $R$, while a blue copy of $K_3$ must have at least two vertices in $B$.
Therefore, if there is a monochromatic  $mK_3$ then

$$m\le\max\left\{\left\lfloor\frac{|R|}{3}\right\rfloor,\left\lfloor\frac{|B|}{2}\right\rfloor\right\}
\le\max\left\{\left\lfloor\frac{3\lfloor(\delta+1)/5\rfloor+2}{3}\right\rfloor,\left\lfloor\frac{2\lfloor(\delta+1)/5\rfloor+1}{2}\right\rfloor\right\}
=\left\lfloor\frac{\delta+1}{5}\right\rfloor.$$

\end{proof}

The next construction and lemma show that, if true, then 
\ref{case:conjBESmedium} of Question~\ref{conjecture:alaBES} is sharp.

\begin{exex}\label{exex:BES2}
For every~$n,\delta\in\mathbb N$ such that $n \geq 25$ and~$4n/5\le \delta\le n-1$, we write~$\text{EX}_{\blacktriangle}^2(n,\delta)$ to denote the following red/blue \hbox{edge-coloured} graph.
We have~$V(\text{EX}_{\blacktriangle}^2(n,\delta))=V_1\dot\cup\dots\dot\cup V_5$ where~$|V_i|=n-\delta$ for every~$2\le i\le 5$ and $|V_1|=4\delta-3n$.
Furthermore,~$V_1=R\dot\cup B$ where $|R|=2\left\lfloor\frac{4\delta-3n+1}{3}\right\rfloor+1$ and $|B|=|V_1|-|R|\geq 0$.
The sets~$V_2,\dots,V_5$ are independent, and
all remaining pairs of vertices form an edge.
The subgraph~$\text{EX}_{\blacktriangle}^2(n,\delta)[V_1,V_2,V_3,V_4,V_5]$ is a blow-up of a badly coloured~$K_5$. 
The edges lying in~$R$ are red, and 
the edges lying in $B$ and the edges incident to both $B$ and $R$ are blue.
See Figure~\ref{fig:EX4} for a representation of~$\text{EX}_{\blacktriangle}^2(n,\delta)$.
\end{exex}

\begin{figure}[h!]
\centering
\begin{tikzpicture}[scale=0.76]

\filldraw[color=red!50,fill=red!50,very thick] (18:3.3)--(90:3.3)--(162:3.3)--(234:3.3)--(306:3.3)--(18:3.3);

\filldraw[color=red!50,fill=white,very thick] (18:2.7)--(90:2.7)--(162:2.7)--(234:2.7)--(306:2.7)--(18:2.7);

\filldraw[color=blue!50,fill=blue!50,very thick] (18:3.6)--(162:3.6)--(162:2.4)--(18:2.4)--(18:3.6);

\filldraw[color=blue!50,fill=blue!50,very thick] (18:3.6)--(234:3.6)--(234:2.4)--(18:2.4)--(18:3.6);

\filldraw[color=blue!50,fill=blue!50,very thick] (90:3.6)--(234:3.6)--(234:2.4)--(90:2.4)--(90:3.6);

\filldraw[color=blue!50,fill=blue!50,very thick] (90:3.6)--(306:3.6)--(306:2.4)--(90:2.4)--(90:3.6);

\filldraw[color=blue!50,fill=blue!50,very thick] (162:3.6)--(306:3.6)--(306:2.4)--(162:2.4)--(162:3.6);

\filldraw[color=black,fill=white,very thick] (18:3) circle (1cm);

\filldraw[color=black,fill=white,very thick] (90:3) circle (1.6cm);

\filldraw[color=black,fill=white,very thick] (162:3) circle (1cm);

\filldraw[color=black,fill=white,very thick] (234:3) circle (1cm);

\filldraw[color=black,fill=white,very thick] (306:3) circle (1cm);

\filldraw[color=blue!50,fill=blue!50,very thick] (100:3.8)--(68:3.3)--(68:2.7)--(100:2.25)--(100:3.8);

\filldraw[color=black,fill=red!50,very thick] (100:3) circle (0.8cm);

\filldraw[color=black,fill=blue!50,very thick] (68:3) circle (0.4cm);

\node at (18:3) {$V_5$};
\node at (162:3) {$V_2$};
\node at (234:3) {$V_3$};
\node at (306:3) {$V_4$};

\node at (100:3) {$R$};
\node at (68:3) {$B$};

\node at (82:4) {$V_1$};
    
\end{tikzpicture}
\caption{A representation of the graph~$\text{EX}_{\blacktriangle}^2(n,\delta)$.
Note that removing all edges inside $V_1=R\dot\cup B$ yields an unbalanced blow-up of $K_5$ with no monochromatic triangle.
The ratio $|B|/|R|$ is approximately $1/2$.}
\label{fig:EX4}
\end{figure}

\begin{lemma}\label{lemma:Bes2}
Let~$n,\delta\in\mathbb N$ such that $n\geq 25$ and~$4n/5\le \delta\le n-1$.
Then~$\text{EX}_{\blacktriangle}^2(n,\delta)$ is an $n$-vertex graph with minimum degree~$\delta(\text{EX}_{\blacktriangle}^2(n,\delta))=\delta$.
Furthermore, for any monochromatic copy of~$mK_3$ in~$\text{EX}_{\blacktriangle}^2(n,\delta)$ we have $m\le\lfloor(4\delta-3n+1)/3\rfloor$.
\end{lemma}

\begin{proof}
We have~$|\text{EX}_{\blacktriangle}^2(n,\delta)|=|V_1|+\ldots+|V_5|=(4\delta-3n)+4(n-\delta)=n$ and the minimum degree is attained by vertices not in $V_1$, yielding $\delta(\text{EX}_{\blacktriangle}^2(n,\delta))=\delta$.
Also, note that
\begin{align*}
\left(2\left\lfloor\frac{4\delta-3n+1}{3}\right\rfloor+1\right)+\left\lfloor\frac{4\delta-3n+1}{3}\right\rfloor 
& = 3\left\lfloor\frac{4\delta-3n+1}{3}\right\rfloor+1 \\
& \ge3\left(\frac{4\delta-3n-1}{3}\right)+1=4\delta-3n=|V_1|.
\end{align*}
Since $|R|+|B|=|V_1|$, it follows that $|B|\le\lfloor(4\delta-3n+1)/3\rfloor$.

As~$\text{EX}_{\blacktriangle}^2(n,\delta)[V_1,V_2,V_3,V_4,V_5]$ is a blow-up of a badly coloured~$K_5$, every monochromatic  $K_3$ in $\text{EX}_{\blacktriangle}^2(n,\delta)$ must have at least two vertices in $V_1$.
Therefore, $m$ is bounded above by the size of the largest monochromatic matching in $V_1$.
Observe that all red edges in $V_1$ lie in $R$, whereas all blue edges in $V_1$ are incident to $B$.
Therefore, if there is a monochromatic $mK_3$ then
$$m\le\max\left\{\left\lfloor\frac{|R|}{2}\right\rfloor,|B|\right\}=\left\lfloor\frac{4\delta-3n+1}{3}\right\rfloor ,$$
as required.
\end{proof}

Our final construction shows the sharpness of part~\ref{case:conjBESsmall} of Question~\ref{conjecture:alaBES} and thus of part~\ref{case:BESsmall} of Theorem~\ref{theorem:alaBES}.

\begin{exex}\label{exex:}
For every~$n,\delta\in\mathbb N$ such that~$4n/5\le \delta \leq n-1$, we write~$\text{EX}_{\blacktriangle}^3(n,\delta)$ to denote the following red/blue \hbox{edge-coloured} graph.
We have~$V(\text{EX}_{\blacktriangle}^3(n,\delta))=V_1\dot\cup\dots\dot\cup V_5$ where~$|V_i|=n-\delta$ for every~$2\le i\le 5$ and $|V_1|=4\delta-3n$.
Furthermore,~$V_1=R\dot\cup B\dot\cup S$ where $|S|=n-\delta$, $|R|=\left\lceil\frac{5\delta-4n}{2}\right\rceil$ and $|B|=\left\lfloor\frac{5\delta-4n}{2}\right\rfloor$.
The sets~$V_2,\dots,V_5$ and $S$ are independent, and 
all other  pairs of vertices form an edge.
The subgraph~$\text{EX}_{\blacktriangle}^3(n,\delta)[V_1,V_2,V_3,V_4,V_5]$ is a blow-up of a badly coloured~$K_5$.
The edges in~$R\cup  S$ are red, and
the edges incident to $B$ in $V_1$ are blue.
See Figure~\ref{fig:EX5} for a representation of~$\text{EX}_{\blacktriangle}^3(n,\delta)$.
\end{exex}

\begin{figure}[h!]
\centering
\begin{tikzpicture}[scale=0.76]

\filldraw[color=red!50,fill=red!50,very thick] (18:3.3)--(90:3.3)--(162:3.3)--(234:3.3)--(306:3.3)--(18:3.3);

\filldraw[color=red!50,fill=white,very thick] (18:2.7)--(90:2.7)--(162:2.7)--(234:2.7)--(306:2.7)--(18:2.7);

\filldraw[color=blue!50,fill=blue!50,very thick] (18:3.6)--(162:3.6)--(162:2.4)--(18:2.4)--(18:3.6);

\filldraw[color=blue!50,fill=blue!50,very thick] (18:3.6)--(234:3.6)--(234:2.4)--(18:2.4)--(18:3.6);

\filldraw[color=blue!50,fill=blue!50,very thick] (90:3.6)--(234:3.6)--(234:2.4)--(90:2.4)--(90:3.6);

\filldraw[color=blue!50,fill=blue!50,very thick] (90:3.6)--(306:3.6)--(306:2.4)--(90:2.4)--(90:3.6);

\filldraw[color=blue!50,fill=blue!50,very thick] (162:3.6)--(306:3.6)--(306:2.4)--(162:2.4)--(162:3.6);

\filldraw[color=black,fill=white,very thick] (18:3) circle (1cm);

\filldraw[color=black,fill=white,very thick] (90:3) circle (1.6cm);

\filldraw[color=black,fill=white,very thick] (162:3) circle (1cm);

\filldraw[color=black,fill=white,very thick] (234:3) circle (1cm);

\filldraw[color=black,fill=white,very thick] (306:3) circle (1cm);

\filldraw[color=red!50,fill=red!50,very thick] (107:2.5)--(90:3.6)--(90:4)--(107:2.9)--(107:2.5);

\filldraw[color=blue!50,fill=blue!50,very thick] (73:2.5)--(90:3.6)--(90:4)--(73:2.9)--(73:2.5);

\filldraw[color=blue!50,fill=blue!50,very thick] (73:2.5)--(107:2.5)--(107:2.9)--(73:2.9)--(73:2.5);

\filldraw[color=black,fill=red!50,very thick] (107:2.7) circle (0.6cm);

\filldraw[color=black,fill=blue!50,very thick] (73:2.7) circle (0.6cm);

\filldraw[color=black,fill=white,very thick] (90:3.8) circle (0.6cm);

\node at (18:3) {$V_5$};
\node at (162:3) {$V_2$};
\node at (234:3) {$V_3$};
\node at (306:3) {$V_4$};

\node at (107:2.7) {$R$};
\node at (73:2.7) {$B$};

\node at (90:3.8) {$S$};

\node at (73:3.8) {$V_1$};
    
\end{tikzpicture}
\caption{A representation of the graph~$\text{EX}_{\blacktriangle}^3(n,\delta)$.
Note that removing all edges inside $V_1=R\dot\cup B\dot\cup S$ yields an unbalanced blow-up of $K_5$ with no monochromatic triangle.}
\label{fig:EX5}
\end{figure}

\begin{lemma}\label{lemma:exeBes3}
Let~~$n,\delta\in\mathbb N$ such that~$4n/5\le \delta \leq n-1$.
Then~$\text{EX}_{\blacktriangle}^3(n,\delta)$ is an~$n$-vertex graph with minimum degree~$\delta(\text{EX}_{\blacktriangle}^3(n,\delta))=\delta$.
Furthermore, for any monochromatic copy of~$mK_3$ in~$\text{EX}_{\blacktriangle}^3(n,\delta)$ we have $m\le\lceil(5\delta-4n)/2\rceil$.
\end{lemma}

\begin{proof}
We have~$|\text{EX}_{\blacktriangle}^3(n,\delta)|=|V_1|+\ldots+|V_5|=(4\delta-3n)+4(n-\delta)=n$, and the minimum degree is
$\delta(\text{EX}_{\blacktriangle}^3(n,\delta))=\delta$, attained by the vertices in $V_2\cup \ldots \cup V_5\cup S$.

As~$\text{EX}_{\blacktriangle}^3(n,\delta)[V_1,V_2,V_3,V_4,V_5]$ is a blow-up of a badly coloured~$K_5$, every monochromatic  $K_3$ in $\text{EX}_{\blacktriangle}^3(n,\delta)$ must have at least two vertices in $V_1$.
Therefore, $m$ is bounded above by the size of the largest monochromatic matching in $V_1$.
Observe that all red edges in $V_1$ are adjacent to $R$, whereas all blue edges in $V_1$ are incident to $B$.
Hence, any monochromatic matching has size at most $\max\{|R|,|B|\}=\left\lceil\frac{5\delta-4n}{2}\right\rceil$, implying $m\le\left\lceil\frac{5\delta-4n}{2}\right\rceil$ as required.
\end{proof}


\section{Proof of Theorem~\ref{theorem:alaMoon}}

The proofs of parts~\ref{case:Moonlarge},~\ref{case:Moonmedium} and~\ref{case:Moonsmall} of Theorem~\ref{theorem:alaMoon} follow a common strategy.
We first partition the vertex set of the host graph into \hbox{vertex-disjoint} cliques of appropriate size and then find monochromatic copies of~$K_3$ within each clique.
The first step is essentially achieved by applying the \hbox{Hajnal--Szemer\'edi} theorem, which we now state.

\begin{theorem}[Hajnal--Szemer\'edi theorem ~\cite{HajnalS}]\label{thm:HajnalSzm}
Let~$n,t\in\mathbb N$ such that~$t$ divides~$n$.
If~$G$ is a graph on~$n$ vertices with~$\delta(G)\ge(1-1/t)n$ then~$G$ contains a perfect~\hbox{$K_t$-tiling}.
\end{theorem}

It will be convenient to use the following apparently stronger but in fact equivalent statement to the
Hajnal--Szemer\'edi theorem. While it is a well-known statement, for the sake of completeness we show how to deduce it from the Hajnal--Szemer\'edi theorem.

\begin{theorem}\label{thm:HajnalSzm2}
Let~$n,t\in\mathbb N$ and~$G$ be a graph on~$n$ vertices such that 
$$\left(1-\frac{1}{t-1}\right)n\le\delta(G)\le\left(1-\frac{1}{t}\right)n.$$
Then~$G$ contains a~\hbox{$K_t$-tiling} consisting of~$(t-1)\delta(G)-(t-2)n$ copies of~$K_t$ and a~\hbox{$K_{t-1}$-tiling} consisting of~$(t-1)n-t\delta(G)$ copies of~$K_{t-1}$, such that the two tilings are \hbox{vertex-disjoint}.
\end{theorem}

\begin{proof}
Let~$G'$ be the graph obtained by adding a set~$S$ of~$(t-1)n-t\delta(G)\ge0$ new vertices to~$G$ and all edges with exactly one vertex in~$S$.
Then~$G'$ is a graph on~$n':=n+|S|=t(n-\delta(G))$ vertices with minimum degree
\begin{align*}
\delta(G') & =\min\{\delta(G)+|S|,n\}=\min\{(t-1)(n-\delta(G)),n\}=(t-1)(n-\delta(G)).
\end{align*}
In particular,~$\delta(G')=(1-1/t)n'$ and~$n'$ is divisible by~$t$.
By the \hbox{Hajnal--Szemer\'edi} theorem (Theorem~\ref{thm:HajnalSzm}),~$G'$ contains a perfect~\hbox{$K_t$-tiling} consisting of~$n'/t$ copies of~$K_t$.
Observe that no edge lies inside~$S$, thus each copy of~$K_t$ contains at most one vertex in~$S$.
In particular,~$n'/t-|S|$ copies of $K_t$ do not contain a vertex from $S$ and $|S|$ copies of $K_t$ contain exactly one vertex from $S$.

It follows that the original graph~$G$ contains a~\hbox{$K_t$-tiling} consisting of~$n'/t-|S|=(t-1)\delta(G)-(t-2)n$ copies of~$K_t$ and a~\hbox{$K_{t-1}$-tiling} consisting of~$|S|=(t-1)n-t\delta(G)$ copies of~$K_{t-1}$, such that the two tilings are \hbox{vertex-disjoint}.
\end{proof}

At various points of our proofs, we will invoke the
following well-known fact.


\begin{fact}\label{fact:K6}
A~\hbox{$2$-edge-coloured}~$K_6$ contains two monochromatic copies of~$K_3$.   \qed  
\end{fact}

Note that in Fact~\ref{fact:K6} the two copies of~$K_3$ are not necessarily \hbox{vertex-disjoint}.
The next two lemmas assert that, for larger dense graphs, we can indeed guarantee the existence of multiple \hbox{vertex-disjoint} monochromatic copies of~$K_3$.
The first lemma is new, while the second lemma is an immediate corollary of Theorem~\ref{theorem:Moon}.

\begin{lemma}\label{lemma:K7}
A~\hbox{$2$-edge-coloured}~$K_7(2)$ contains three \hbox{vertex-disjoint} monochromatic copies of~$K_3$. 
\end{lemma}

\begin{lemma}[Moon~\cite{Moon}]\label{lemma:K8}
A $2$-edge-coloured $K_8$ contains two vertex-disjoint monochromatic copies of $K_3$. 
\end{lemma}

To apply Lemma~\ref{lemma:K7} in conjunction with Theorem~\ref{thm:HajnalSzm2}, we will
 apply Szemer\'edi's Regularity Lemma~\cite{Szemeredi} and the Blow-up Lemma~\cite{blowup};
the downside of using these techniques is that they cause the~$o(n)$ error term to appear in part~\ref{case:Moonmedium} of Theorem~\ref{theorem:alaMoon}.

In the next three subsections we prove parts~\ref{case:Moonsmall}, \ref{case:Moonlarge} and~\ref{case:Moonmedium} of Theorem~\ref{theorem:alaMoon}, respectively.
The proof of Lemma~\ref{lemma:K7} appears at the end of the section.

\subsection{Proof of Theorem~\ref{theorem:alaMoon}\ref{case:Moonsmall}}

Let~$n\in\mathbb N$ and let $G$ be a $2$-edge-coloured $n$-vertex graph with~$4n/5\le\delta(G)\le 5n/6$.
By Theorem~\ref{thm:HajnalSzm2},~$G$ contains a~$K_6$-tiling consisting of~$5\delta(G)-4n$ copies of~$K_6$.
By Fact~\ref{fact:K6}, each~$K_6$ contains a monochromatic copy of $K_3$.
It follows that~$G$ contains a~\hbox{$K_3$-tiling} consisting of~$5\delta(G)-4n$ monochromatic copies of~$K_3$, as required. \hfill Q.E.D.


\subsection{Proof of Theorem~\ref{theorem:alaMoon}\ref{case:Moonlarge}}

Recall that part~\ref{case:Moonlarge} of Theorem~\ref{theorem:alaMoon} states that any $2$-edge-coloured $n$-vertex graph $G$ with $\delta(G)\ge7n/8$ contains a $K_3$-tiling consisting of $\lfloor(2\delta(G)-n)/3\rfloor$ monochromatic copies of $K_3$. We prove this by induction on $n$.

Before this, we prove the case when $7n/8\le\delta(G)\le(7n+2)/8$ (for all $n\in \mathbb N$).
Note that any induced subgraph~$H$ of~$G$ with~$|H|=8(n-\delta(G))\le n$ satisfies~$\delta(H)\ge\delta(G)-(|G|-|H|)=7(n-\delta(G))=7|H|/8$.
Theorem~\ref{thm:HajnalSzm} implies~$H$, and thus~$G$, contains a~$K_8$-tiling consisting of~$n-\delta(G)$ copies of~$K_8$.
By Lemma~\ref{lemma:K8}, each~$K_8$ contains two \hbox{vertex-disjoint} monochromatic copies of~$K_3$.
Taking the union of all such copies yields a~\hbox{$K_3$-tiling} consisting of precisely~$2(n-\delta(G))$ monochromatic copies of~$K_3$. 
This concludes the verification of this case, as
$$\left\lfloor\frac{2\delta(G)-n}{3}\right\rfloor=2(n-\delta(G))+\left\lfloor\frac{8\delta(G)-7n}{3}\right\rfloor\le2(n-\delta(G))+\left\lfloor\frac{2}{3}\right\rfloor=2(n-\delta(G)).$$

Now we can proceed by induction on $n$. The base cases when $8 \leq n \leq 10$ are covered by the last paragraph.
Next, we check the inductive step. 
Suppose $G$ is an $n$-vertex graph where $n \geq 11$. By the previous paragraph we may assume that~$\delta(G)\ge(7n+3)/8$.
It is easy to show that~$G$ contains a $K_6$ (e.g., by Theorem~\ref{thm:HajnalSzm2}), which in turn contains a monochromatic copy~$T$ of~$K_3$ by Fact~\ref{fact:K6}.
Let~$G':=G\setminus V(T)$.
Note that~$G'$ is a~\hbox{$2$-edge-coloured} graph on~$n-3$ vertices with minimum degree~$\delta(G')\ge\delta(G)-3\ge 7(n-3)/8$.
By the inductive hypothesis,~$G'$ contains a~\hbox{$K_3$-tiling} consisting of~$\lfloor(2\delta(G')-(n-3))/3\rfloor\ge\lfloor(2\delta(G)-n)/3\rfloor-1$ monochromatic copies of~$K_3$.
Adding~$T$ to this tiling yields a~\hbox{$K_3$-tiling} in~$G$ consisting of at least~$\lfloor(2\delta(G)-n)/3\rfloor$ monochromatic copies of~$K_3$.
This concludes the inductive step and the proof.
\hfill Q.E.D.


\subsection{Proof of Theorem~\ref{theorem:alaMoon}\ref{case:Moonmedium}}
To prove this part of the theorem, it suffices to show the following: 
Let~$\eta >0$ and~$n \in \mathbb N$ be sufficiently large. 
Let~$G$ be an $n$-vertex $2$-edge-coloured graph with~$5n/6 \le\delta(G)\leq 7n/8$. 
Then~$G$ contains a~\hbox{$K_3$-tiling} consisting of at least~$(4\delta(G)-3n)/2-\eta n$ monochromatic copies of~$K_3$.

Let $G$ be as in this statement.
First, we use part~\ref{case:Moonsmall} of Theorem~\ref{theorem:alaMoon} to show that we may assume~$\delta(G)$ is bounded away from~$5n/6$.

\begin{claim}\label{claim:bound5/6}
Either~$\delta(G)\ge(5/6+\eta/4)n$ or~$G$ contains a~\hbox{$K_3$-tiling} consisting of at least~$(4\delta(G)-3n)/2-\eta n$ monochromatic copies of~$K_3$.     
\end{claim}

\begin{proof}
Suppose $5n/6 \leq \delta(G)<(5/6+\eta/4)n$.
Then there exists a spanning subgraph~$F$ of~$G$ with~$\delta(F)=\lfloor5n/6\rfloor=n-\lceil n/6\rceil$.
We can therefore apply part~\ref{case:Moonsmall} of Theorem~\ref{theorem:alaMoon} to~$F$.
Thus,~$F$ (and so $G$) contains~$n-5\lceil n/6\rceil\ge n/6-5$ \hbox{vertex-disjoint} monochromatic copies of~$K_3$.
As $\delta(G)<(5/6+\eta/4)n$, we have that 
$(4\delta(G)-3n)/2-\eta n<n/6-5$.
Then  indeed
 $G$ contains a~\hbox{$K_3$-tiling} consisting of at least~$(4\delta(G)-3n)/2-\eta n$ monochromatic copies of~$K_3$, as desired.
\end{proof}

By Claim~\ref{claim:bound5/6}, we may assume~$\delta(G)\ge(5/6+\eta/4)n$.
As  mentioned at the beginning of this section, we now employ Szemer\'edi's Regularity Lemma~\cite{Szemeredi} and the Blow-up Lemma~\cite{blowup}.
Essentially, the former provides an auxiliary graph~$R$ (which is commonly referred to as the {\it reduced graph}) which approximates~$G$ in the following sense: the vertex set of~$G$ can be partitioned into vertex classes~$\{V_v:v\in V(R)\}$ and a small exceptional set $V_0$ such that if~$xy\in E(R)$, then the edges between~$V_x$ and~$V_y$ are evenly distributed.
One can then argue that, using the Blow-up Lemma, for our purposes such evenly distributed edges behave essentially as a complete bipartite graph.
In particular, given a collection of \hbox{vertex-disjoint} cliques in~$R$, one can find \hbox{vertex-disjoint} blow-ups of cliques in the original graph.
Our strategy then is to find an appropriate collection of \hbox{vertex-disjoint} copies of~$K_6$,~$K_7$ and~$K_8$ in the reduced graph, using Theorem~\ref{thm:HajnalSzm2}.
This yields a collection of \hbox{vertex-disjoint} copies of~$K_6(2)$,~$K_7(2)$ and~$K_8(2)$ in the original graph $G$.
We then apply Fact~\ref{fact:K6}, Lemma~\ref{lemma:K8} and, crucially, Lemma~\ref{lemma:K7}.

The next result formalises the ``embedding step" from the reduced graph to the original graph~$G$ described above.
For simplicity of exposition, we avoid introducing the standard notation used for the Regularity Lemma and instead keep the technicalities to a minimum.
We note that the proof of this result is standard; the proof of the statement can be found in the
appendix.

\begin{lemma}[Embedding step]\label{theorem:embedding}
For every~$\eta>0$ there exists~$n_0=n_0(\eta)\in\mathbb{N}$ such that for every graph~$G$ on~$n\geq n_0$ vertices the following holds.
There exist~$m,\ell\in\mathbb{N}$, a partition~$V_0, V_1,\dots, V_\ell$ of~$V(G)$ and a graph~$R$ with vertex set~$\{V_i:i\ge1\}$ such that the following properties  hold:
\begin{enumerate}[label=(\roman*)]
\item~$\delta(G)/n-\eta/4\le\delta(R)/|R|\le\delta(G)/n$; \label{property:i}
\item~$|V_i| = m$ for every~$i\ge1$ and~$|V_0|\leq\eta n/2$; \label{property:ii}
\item If the vertices~$\{V_{i_1},\dots,V_{i_r}\}$ in~$R$ form a clique and~$r\le 8$, then~$G[V_{i_1},\dots,V_{i_r}]$ contains a~$K_r(2)$-tiling consisting of at least~$(1-\eta/2)m/2$ copies of~$K_r(2)$. \label{property:iii}
\end{enumerate}
\end{lemma}

Apply Lemma~\ref{theorem:embedding} to the graph~$G$ to obtain~$m,\ell\in\mathbb N$, a partition~$V_0,V_1,\dots,V_\ell$ of~$V(G)$ and a graph~$R$ satisfying properties~\ref{property:i}--\ref{property:iii} of Lemma~\ref{theorem:embedding}.

\begin{claim}\label{claim:caseb}
There exists a~\hbox{$K_3$-tiling} in~$G$ consisting of at least~$(4\delta(R)-3|R|)\cdot\frac{(1-\eta/2)m}{2}$ monochromatic copies of~$K_3$.    
\end{claim}
\begin{proof}
By property~\ref{property:i} and the fact that~$(5/6+\eta/4)n\le\delta(G)\le 7n/8$, it follows that
$$5|R|/6\le\delta(R)\le 7|R|/8.$$  

Suppose first that~$\delta(R)\le 6|R|/7$.
Then by Theorem~\ref{thm:HajnalSzm2},~$R$ contains a~$K_7$-tiling consisting of~$6\delta(R)-5|R|$ copies of~$K_7$ and a~$K_6$-tiling consisting of~$6|R|-7\delta(R)$ copies of~$K_6$ such that the two tilings are \hbox{vertex-disjoint}.
By property~\ref{property:iii},~$G$ contains a~$K_7(2)$-tiling consisting of at least~$(6\delta(R)-5|R|)\frac{(1-\eta/2)m}{2}$ copies of~$K_7(2)$ and a~$K_6(2)$-tiling consisting of at least~$(6|R|-7\delta(R))\frac{(1-\eta/2)m}{2}$ copies of~$K_6(2)$ such that the two tilings are \hbox{vertex-disjoint}.
Fact~\ref{fact:K6} implies that every copy of~$K_6(2)$ contains two \hbox{vertex-disjoint} monochromatic copies of~$K_3$, whereas Lemma~\ref{lemma:K7} implies every copy of~$K_7(2)$ contains three \hbox{vertex-disjoint} monochromatic copies of~$K_3$.
It follows that~$G$ contains a~\hbox{$K_3$-tiling} consisting of at least
$$\left(3\cdot(6\delta(R)-5|R|)+2\cdot(6|R|-7\delta(R))\right)\cdot\frac{(1-\eta/2)m}{2}=(4\delta(R)-3|R|)\cdot\frac{(1-\eta/2)m}{2}$$
monochromatic copies of~$K_3$, as required.

The case~$\delta(R) \geq 6|R|/7$ is very similar.
By Theorem~\ref{thm:HajnalSzm2},~$R$ contains a~$K_8$-tiling consisting of~$7\delta(R)-6|R|$ copies of~$K_8$ and a~$K_7$-tiling consisting of~$7|R|-8\delta(R)$ copies of~$K_7$ such that the two tilings are \hbox{vertex-disjoint}.
It follows that~$G$ contains a~$K_8(2)$-tiling consisting of at least~$(7\delta(R)-6|R|)\frac{(1-\eta/2)m}{2}$ copies of~$K_8(2)$ and a~$K_7(2)$-tiling consisting of at least~$(7|R|-8\delta(R))\frac{(1-\eta/2)m}{2}$ copies of~$K_7(2)$ such that the two tilings are \hbox{vertex-disjoint}.
Lemma~\ref{lemma:K8} implies every copy of~$K_8(2)$ contains four \hbox{vertex-disjoint} monochromatic copies of~$K_3$, whereas Lemma~\ref{lemma:K7} implies every copy of~$K_7(2)$ contains three \hbox{vertex-disjoint} monochromatic copies of~$K_3$.
It follows that~$G$ contains a~\hbox{$K_3$-tiling} consisting of at least
\begin{align*} 
\left(3\cdot(7|R|-8\delta(R))+4\cdot(7\delta(R)-6|R|)\right)\cdot\frac{(1-\eta/2)m}{2}=(4\delta(R)-3|R|)\cdot\frac{(1-\eta/2)m}{2}
\end{align*}
monochromatic copies of~$K_3$.
\end{proof}

Note that~$m=(n-|V_0|)/|R|$ and thus~$m\ge(1-\eta/2)n/|R|$ by property~\ref{property:ii}.
Furthermore,~$\delta(R)/|R|\ge\delta(G)/n-\eta/4$ by property~\ref{property:i}.
Using these inequalities, we obtain
\begin{align*} 
& (4\delta(R)-3|R|)\cdot\frac{(1-\eta/2)m}{2}
\ge(4\delta(R)-3|R|)\cdot\frac{(1-\eta/2)^2n}{2|R|} 
= (1-\eta/2)^2\cdot(4\delta(R)/|R|-3)\frac{n}{2} \\
 \ge\,\, & (1-\eta/2)^2\cdot\frac{(4\delta(G)-\eta n-3n)}{2}
\ge \frac{4\delta(G)-3n}{2}-\eta n.
\end{align*}

Therefore, by Claim~\ref{claim:caseb},~$G$ contains a~\hbox{$K_3$-tiling} consisting of at least~$(4\delta(G)-3n)/2-\eta n$ monochromatic $K_3$.
This concludes the proof of case~\ref{case:Moonmedium}. \hfill Q.E.D.

\subsection{Proof of Lemma~\ref{lemma:K7}}

We start with the following claim.

\begin{claim}\label{claim:K7}
A~\hbox{$2$-edge-coloured}~$K_7$ contains two monochromatic~$K_3$ sharing at most one vertex.    
\end{claim}

\begin{proof}
Suppose the statement of the claim does not hold.
Let~$G$ be a~\hbox{$2$-edge-coloured} complete graph with $V(G)=\{v_1,\dots,v_7\}$ such that every pair of monochromatic copies of~$K_3$ share two vertices. 
By Fact~\ref{fact:K6}, there are two monochromatic copies of~$K_3$ in~$G$,
without loss of generality~$T_1=v_1v_2v_3$ and~$T_2=v_1v_2v_4$.
Note that the only  copy of~$K_3$ in~$G\setminus\{v_1\}$ that shares two vertices with both~$T_1$ and~$T_2$ is~$v_2v_3v_4$.
In particular,~$G\setminus\{v_1\}$ contains at most one monochromatic~$K_3$.
Since~$G\setminus\{v_1\}$ is a copy of~$K_6$, this contradicts 
Fact~\ref{fact:K6}.
\end{proof}

Let~$G$ be a~\hbox{$2$-edge-coloured}~$K_7(2)$ with vertex set~$U\dot\cup V$ where~$U=\{u_1,\dots,u_7\}$,~$V=\{v_1,\dots,v_7\}$ and the non-edges of~$G$ are precisely the pairs~$u_iv_i$ for~$i\in [7]$.
By Claim~\ref{claim:K7}, there are two monochromatic copies of~$K_3$ in~$G[V]$ sharing at most one vertex.
If they are vertex-disjoint we are done, thus we may assume that~$v_1v_2v_3$ and~$v_3v_4v_5$ are monochromatic copies of~$K_3$.
By Fact~\ref{fact:K6}, the graph~$G[U\setminus\{u_3\}]$ contains a monochromatic copy of~$K_3$, say~$u_iu_ju_k$.
By the pigeonhole principle, we have either~$|\{1,2\}\cap\{i,j,k\}|\le1$ or~$|\{4,5\}\cap\{i,j,k\}|\le1$.
Without loss of generality we may  assume~$|\{1,2\}\cap\{i,j,k\}|\le1$ and in particular~$1\notin\{i,j,k\}$.
Let~$S:=\{u_1,u_3,v_4,v_5,v_6,v_7\}$.
Observe that~$S$ is \hbox{vertex-disjoint} to~$v_1v_2v_3$ and~$u_iu_ju_k$.
Furthermore,~$G[S]$ is a copy of~$K_6$ and thus it contains a monochromatic copy~$T$ of~$K_3$ by Fact~\ref{fact:K6}.
Note that~$T$,~$v_1v_2v_3$ and~$u_iu_ju_k$ are three \hbox{vertex-disjoint} monochromatic copies of~$K_3$, as required.\hfill Q.E.D.

\section{Proof of Theorem~\ref{theorem:alaBES}}


Let $n\in\mathbb N$ and $G$ be a $2$-edge-coloured $n$-vertex graph.
If $4n/5\le\delta(G)\le 5n/6$ then by part~\ref{case:Moonsmall} of Theorem~\ref{theorem:alaMoon} there is a~\hbox{$K_3$-tiling} in $G$ consisting of at least~$5\delta(G)-4n$ monochromatic $K_3$.
At least~$m:=\left\lceil (5\delta(G)-4n)/2\right\rceil$ of these triangles receive the same colour, and thus they form a monochromatic copy of~$mK_3$. 
This verifies part~\ref{case:BESsmall} of Theorem~\ref{theorem:alaBES}.

For part~\ref{case:BESlarge}, a  different approach is needed.
We use the following definition which was introduced in~\cite{BurrES} for the proof of Theorem~\ref{theorem:BES}.

\begin{define}\label{define:bowtie}
A {\it bowtie} consists of two monochromatic copies of $K_3$ of different colours which share exactly one vertex. 
\end{define}

A useful fact, observed in~\cite{BurrES}, is that if a complete graph contains two vertex-disjoint monochromatic copies of $K_3$ of different colours then it must contain a bowtie.
The following lemma is a strengthening of this statement.

\begin{lemma}\label{lemma:bowtie1}
Suppose a $2$-edge-coloured $K_6$ contains two vertex-disjoint monochromatic copies of $K_3$ of different colours. Then
for every vertex $v\in V(K_6)$, there exists a bowtie containing $v$.
\end{lemma}
\begin{proof}
Without loss of generality, we may assume $V(K_6)=[6]$, $123$ is a red $K_3$ and $456$ is a blue $K_3$.
By symmetry, it suffices to prove the statement of the lemma for $v=1$.

If $1$ is incident to two blue edges, $14$ and $15$ say, then the copies  $123$ and $145$ of $K_3$ form a bowtie containing $1$.
Thus, $1$ is incident to at most one blue edge.
Similarly, $2$ is incident to at most one blue edge.
It follows that for some $i\in\{4,5,6\}$ the edges $1i$ and $2i$ are red.
Then the copies $12i$ and $456$ of $K_3$ form a bowtie containing $1$.
\end{proof}

Using Lemma~\ref{lemma:bowtie1}, we obtain the following.

\begin{lemma}\label{lemma:bowtie2}
Suppose a $2$-edge-coloured $K_7$ contains a bowtie.
Then there exists another bowtie on a different vertex set.  
\end{lemma}

\begin{proof}
Let $B$ be a bowtie in $K_7$ and let $K_B$ denote the blue copy of $K_3$ in $B$.
Let $\{x,y\}=V(K_7)\setminus V(B)$.
It suffices to show that there exists a monochromatic copy $K$ of $K_3$ containing either $x$ or $y$ (or both).
Indeed, suppose such $K$ exists and without loss of generality suppose that $K$ is red. 
If $K$ and $K_B$ are disjoint, then Lemma~\ref{lemma:bowtie1} implies that there is a bowtie $B'$ containing either $x$ or $y$;
so $B$ and $B'$ have different vertex sets. If $K$ and $K_B$ intersect, then they must share precisely one vertex; so $K$ and $K_B$ form a bowtie on a different vertex set to $B$.

We now prove that $K$ exists.
Without loss of generality, we may assume $xy$ is red.
If there is a vertex $z\in V(K_B)$ such that $xz$ and $yz$ are red, we are done.
Thus, for every $z\in V(K_B)$, there is a blue edge incident to $z$ and $xy$;
so two vertices of $K_B$ must be both adjacent via blue edges to some vertex  $w \in \{x,y\}$, and so we are done.
\end{proof}

We are now ready to prove part~\ref{case:BESlarge}.
Let $n\in\mathbb N$ and $G$ be a $2$-edge-coloured $n$-vertex graph with $\delta(G)\ge 65n/66$.
Set $m:=\lfloor(\delta(G)+1)/5\rfloor$.

Let $\mathcal B$ and $\mathcal T$ be two collections of subsets of $V(G)$ satisfying the following properties.

\medskip

\noindent\underline{\bf Properties:}

\smallskip

\begin{enumerate}[label=(\roman*)]

\item  For every distinct $X,Y\in \mathcal B\cup \mathcal T$ we have $X\cap Y= \emptyset$.

\item Each $X \in \mathcal B$ induces a copy of $K_5$ in $G$ that contains a bowtie.

\item Each $X \in \mathcal T$ induces a  monochromatic copy of $K_3$ in $G$. Moreover,  all these monochromatic copies of $K_3$ have the same colour.

\item $|\mathcal B|$ is as large as possible.
Conditioned on this, $|\mathcal T|$ is as large as possible.
\end{enumerate}

\smallskip

It is easy to see that $G$ contains a monochromatic copy of $(|\mathcal B|+|\mathcal T|)K_3$: if the copies of $K_3$ obtained from  $\mathcal T$ are  red say, then we select a red $K_3$ in each element of $\mathcal B$ and $\mathcal T$ and then take their disjoint union. 
Hence, it suffices to show that $|\mathcal B|+|\mathcal T|\ge m$.
We assume for a contradiction that $|\mathcal B|+|\mathcal T|<m$.
We  abuse notation slightly and write $V(\mathcal B)$, $V(\mathcal T)$ and $V(\mathcal B\cup\mathcal T)$ to denote the number of vertices covered by the elements of $\mathcal B$, $\mathcal T$ and $\mathcal B\cup\mathcal T$ respectively.
 
We start by providing a lower bound on $|\mathcal B|$.

\begin{claim}\label{claim:boundB}
We have $|\mathcal B|\ge 5n/33$.    
\end{claim}

\begin{proof}
Suppose that $|\mathcal B|<5n/33$.
Then the number of vertices in $V(\mathcal B\cup\mathcal T)$ is 
$$5|\mathcal B|+3|\mathcal T|\le 5|\mathcal B|+3(m-1-|\mathcal B|)=3m+2|\mathcal B|-3
< \frac{3 \delta (G)}{5}+\frac{10n}{33}.
%
$$
First suppose that $\mathcal T=\emptyset$.
Since $\delta(G)\ge 65n/66$, we have that 
$\delta(G \setminus V(\mathcal B))\geq 65n/66-5|\mathcal B|> 4n/5-4
|\mathcal B|
= 4|G \setminus V(\mathcal B)|/5$, where the second inequality follows as 
we assume that $|\mathcal B|<5n/33$;
hence,
there  exists a $K_6$ in $G\setminus V(\mathcal B)$.
However, this copy  of $K_6$ must contain a monochromatic $K_3$ that is vertex-disjoint to $V(\mathcal B)$, contradicting the assumption that $\mathcal T=\emptyset$.

Suppose now that $\mathcal T\neq\emptyset$, say $\mathcal T$ contains a set inducing a red copy $T$ of $K_3$. 
Let $n':=n-| V(\mathcal B\cup\mathcal T)|$.
Since $\delta(G)\ge 65n/66$ and 
$|V(\mathcal B\cup\mathcal T)|<{3 \delta (G)}/{5}+{10n}/{33}$, 
we have that every vertex  $x \in V(G)$ has at least $\delta (G)-| V(\mathcal B\cup\mathcal T)|> 6n'/7$ neighbours in $G$ that lie in 
$V(G)\setminus V(\mathcal B\cup\mathcal T)$.
Thus,
there  exists a  $K_5$ in $G\setminus V(\mathcal B\cup\mathcal T)$ that together with $T$ forms a  $K_8$ in $G$.
If this copy of $K_8$ contains a blue $K_3$ then it contains a bowtie by Lemma~\ref{lemma:bowtie1}.
This contradicts the maximality of $\mathcal B$, thus any monochromatic $K_3$ in this copy of $K_8$ must be red.
By Lemma~\ref{lemma:K8}, this $K_8$ contains two vertex-disjoint red copies of $K_3$.
This contradicts the maximality of $\mathcal T$.
\end{proof}

Combining the lower bound in Claim~\ref{claim:boundB} with the minimum degree condition, we obtain the following claim.

\begin{claim}\label{claim:key}
Let $S\subseteq V(G)\setminus V(\mathcal B)$ such that $|S|\le 10$.
Then there exists $B\in\mathcal B$ such that the graph 
$G[B\cup\{s\}]$ is complete for every $s\in S$.
\end{claim}

\begin{proof}
Suppose for a contradiction the claim is false; so there does not exist a set 
$B\in\mathcal B$ such that, in $G$, the vertices in $S\subseteq  V(G)\setminus V(\mathcal B)$ are adjacent to every vertex in $B$.
By the pigeonhole principle, this implies that there is a vertex $s \in S $ that is non-adjacent to at least $|\mathcal B|/|S|\geq n/66$ vertices in $G$ (and itself), a contradiction as $\delta(G)\ge 65n/66$. 
\end{proof}

We are now ready to combine all our lemmas and claims to conclude the proof.


Let $X:=\emptyset$. If $\mathcal T$ is non-empty, let $T$ be an element of $\mathcal T$.
Otherwise, set $T:=\emptyset$.
We iterate the following procedure as long as there is some edge $uv$ in $G$ non-incident to $V(\mathcal B\cup\mathcal T)\cup X$ and~$|X|\le 5$.
During the procedure, we maintain the property that  $G[T\cup X]$ is a clique.

\medskip

\noindent\underline{\bf Procedure:} 
Apply Claim~\ref{claim:key} with $S:=T\cup X\cup\{u,v\}$
to find some $B\in\mathcal B$ such that $G[B\cup\{s\}]$ is complete for every $s\in S$.
In particular, $G[B\cup\{u,v\}]$ is a copy of $K_7$.

By Lemma~\ref{lemma:bowtie2}, there exists a vertex set $B'\subseteq B\cup\{u,v\}$ such that $B'$ spans a bowtie and $B'\neq B$. Let $z \in B \setminus B'$.

Set $\mathcal B:=(\mathcal B\setminus \{B\}) \cup \{B'\}$ and $X:=X\cup\{z\}$.
Note that $\mathcal B$ and $\mathcal T$ still satisfy the initial properties.
Furthermore, $G[T\cup X]$ is a clique.\\
END PROCEDURE

\medskip

If at the end of this procedure $|X|=6$, then define $Y:=X$.
If $|X|\leq 5$ and there does not exist $w \in V(G)\setminus (V(\mathcal B\cup\mathcal T)\cup X)$ such that 
$G[T \cup X\cup \{w\}]$ is a clique, then we set $Y:=X$.
Otherwise, there is 
a $w \in V(G)\setminus (V(\mathcal B\cup\mathcal T)\cup X)$ such that 
$G[T \cup X\cup \{w\}]$ is a clique, however, in $G$, $w$ is not adjacent to any vertex in 
$V(G)\setminus (V(\mathcal B\cup\mathcal T)\cup X)$;\footnote{This latter condition follows by definition of the procedure above.}
in this case we define $Y:=X\cup \{w\}$.

Note that in all cases $G[T \cup Y]$ is a clique.
Further, if $|Y|\leq 5$ then there are no edges in
$G\setminus (V(\mathcal B\cup\mathcal T)\cup Y)$.

\begin{claim}
Either  (a) $|Y|= 6$ or (b) $|Y|=5$ and $|\mathcal B|+|\mathcal T|=m-1$.    
\end{claim}
\begin{proof}
If~$|Y|= 6$ we are done, so suppose that~$|Y|\le 5$.
If $V(G) = V(\mathcal B\cup\mathcal T)\cup Y$,
then $\delta(G) \leq |V(\mathcal B\cup\mathcal T)\cup Y|-1$.
Otherwise, by definition of $Y$, for
 every vertex $x$ in $V(G)\setminus(V(\mathcal B\cup\mathcal T)\cup Y)$ we have
$\delta(G)\le d_G(x)\le |V(\mathcal B\cup\mathcal T)\cup Y|-1.$

In both cases we conclude that 
$$\delta(G)\le |V(\mathcal B\cup\mathcal T)\cup Y|-1\le5(|\mathcal B|+|\mathcal T|)+|Y|-1.$$
Combining the above with $|\mathcal B|+|\mathcal T|\leq m-1$ and $m\le(\delta(G)+1)/5$ we obtain
$$\delta(G)\le5m-6+|Y|\le\delta(G)-5+|Y|.$$
In particular, we must have $|Y|\ge5$.
However, we assumed that~$|Y|\le5$.
Therefore, $|Y|=5$ and all the above inequalities are in fact equalities.
Thus, we have $|\mathcal B|+|\mathcal T|=m-1$.
\end{proof}

First, suppose $\mathcal T$ is non-empty, hence $G[T]$ is a monochromatic $K_3$.
We have that $G[T\cup Y]$ is a clique, and in particular it has at least $|T|+|Y|\ge3+5=8$ vertices.
If $G[T\cup Y]$ contains a monochromatic $K_3$ whose colour is different from the colour of $G[T]$, then by Lemma~\ref{lemma:bowtie1} it contains a bowtie.
This contradicts the assumption that $\mathcal B$ is maximal.
Hence, all monochromatic copies of $K_3$ in $G[T\cup Y]$ must be of the same colour as $G[T]$.
By Lemma~\ref{lemma:K8}, $G[T\cup Y]$ must contain two vertex-disjoint monochromatic copies of $K_3$.
This contradicts the assumption that $\mathcal T$ is maximal.

Therefore,  it must be the case that $\mathcal T$ is empty.
If $|Y|=6$, then $G[Y]$ contains a monochromatic $K_3$ that does not intersect $V(\mathcal B)$.
This contradicts the maximality of $\mathcal T$.
Hence, we may assume that $|Y|=5$ and so $|\mathcal B|+|\mathcal T|=|\mathcal B|=m-1$.
By applying Claim~\ref{claim:key} with $S=Y$, there is some $B$ in $\mathcal B$ such that $G[B\cup Y]$ is a clique.
Since $|B|+|Y|=5+5=10$, by Theorem~\ref{theorem:BES}, $G[B\cup Y]$ contains a monochromatic copy of $2K_3$, say red.
It follows that there are $|\mathcal B\setminus \{B\}|+2$  vertex-disjoint red copies of $K_3$.
Since $|\mathcal B|=m-1$, it follows that there is a  red copy of $mK_3$, as required.\hfill Q.E.D.

\section{Further results and concluding remarks}

\subsection{Further results on Problem~\ref{problem:alaMoon}}

In this subsection, we discuss some further results related to Problem~\ref{problem:alaMoon}. 
First, similarly to Theorems~\ref{theorem:alaBES} and~\ref{theorem:alaMoon},  if $\delta$ is not large enough then we might not be able to ensure even a single monochromatic copy of the sought structure.
This is formalised via the \emph{chromatic Ramsey number}, a parameter introduced in~\cite{ars}.

\begin{define}\label{define:degRamsey}
For $r\in\mathbb N$ and a graph $H$,
we say a graph $G$ is \emph{$(H,r)$-Ramsey} if every $r$-edge-colouring of $G$ contains a monochromatic copy of $H$. The \emph{chromatic Ramsey number} $R_{\chi} (H,r)$ is the least $m \in \mathbb N$ such that there exists an $(H,r)$-Ramsey graph of chromatic number $m$.
\end{define}
For example, it is simple to see that
$R_\chi (K_\ell,r)=R_r (K_\ell)$ for all $\ell,r\geq 2$.

Note that by definition of 
$R_\chi (H,r)$, 
there is an $r$-edge-colouring of the $(R_\chi (H,r)-1)$-partite Tur\'an graph on $n$ vertices (for any choice of $n \in \mathbb N$)
that does not yield a monochromatic copy of $H$. Thus, 
Problems~\ref{problem:alaBES} and~\ref{problem:alaMoon} are trivial if
\begin{equation}\label{eq:trivial}
\delta\le\left\lfloor\left(1-\frac{1}{R_\chi(H,r)-1}\right)n\right\rfloor,
\end{equation}
since, in this case, we cannot guarantee even a single monochromatic copy of~$H$. On the other hand, the Erd\H{o}s--Stone--Simonovits theorem implies that for any $\eta>0$, every sufficiently large $r$-edge-coloured $n$-vertex graph $G$ with $\delta(G)\geq (1-\frac{1}{R_\chi(H,r)-1}+\eta )n$ does contain a monochromatic copy of $H$.

\medskip

The next simple result generalises part \ref{case:Moonsmall}  of Theorem~\ref{theorem:alaMoon} for larger cliques and multiple colours.
Set $R_r(\ell):=R_r(K_{\ell})$.

\begin{theorem}\label{theorem:further1}
Let~$n,r\in\mathbb N$, $\ell \geq 2$, and let $G$ be an $r$-edge-coloured $n$-vertex graph such that
$$\left(1-\frac{1}{R_r(\ell)-1}\right)n\le\delta(G)\le \left(1-\frac{1}{R_r(\ell)}\right)n.$$
Then there exists a \hbox{$K_\ell$-tiling} in~$G$ such that every copy of~$K_\ell$ is monochromatic and the number of copies of~$K_\ell$ in the tiling is at least
$(R_r(\ell)-1)\delta(G)-(R_r(\ell)-2)n$.
\end{theorem}

\begin{proof} 
By Theorem~\ref{thm:HajnalSzm2},~$G$ contains a~$K_{R_r(\ell)}$-tiling consisting of~$(R_r(\ell)-1)\delta(G)-(R_r(\ell)-2)n$ copies of~$K_{R_r(\ell)}$.
By definition, each~$K_{R_r(\ell)}$ contains a monochromatic copy of $K_\ell$.
It follows that~$G$ contains a~\hbox{$K_\ell$-tiling} consisting of~$(R_r(\ell)-1)\delta(G)-(R_r(\ell)-2)n$ monochromatic copies of~$K_\ell$, as required.  
\end{proof}

Note that Theorem~\ref{theorem:further1} is best possible.
Indeed, consider an $r$-edge-coloured $K_{R_r(\ell)}$ such that all monochromatic copies of $K_\ell$ have a common vertex $v$. Let $\delta \in \mathbb N$ so that $(1-\frac{1}{R_r(\ell)-1})n< \delta\le (1-\frac{1}{R_r(\ell)})n.$
Consider the $n$-vertex blow-up $G$ of this $K_{R_r(\ell)}$ where $v$ is replaced by a class $V$ of size $(R_r(\ell)-1)\delta-(R_r(\ell)-2)n$ and all other vertices are replaced by a class of size $n-\delta$.
Note that $\delta (G)= \delta$, and every monochromatic copy of $K_\ell$ in $G$ must contain some vertex in $V$.
Therefore, any collection of more than $|V|$ monochromatic copies of $K_\ell$ cannot  form a $K_\ell$-tiling.

\smallskip

The next theorem considers the case where $\delta(G)$ is close to $n-1$.
To state it, we need the following variant of the Ramsey number.

\begin{define}
Let  $r\ge2$.
A {\it special $r$-edge-colouring} of a graph $G$ is an $r$-edge-colouring of $G$ using colours $c_1,\dots, c_r$ such that there exists a  $v \in V(G)$ that is not incident to any edge of colour $c_i$, for some $i \in [r]$.
The \emph{special Ramsey number} $SR_r(H)$ is the smallest  $n\in \mathbb N$ such that any special $r$-colouring of $K_n$ yields a monochromatic copy of $H$.
We write $SR_r(\ell):=SR_r(K_\ell)$.
\end{define}

\begin{theorem}\label{further2}
For every $r,\ell\ge2$, there exists $n_0\in\mathbb N$ such that the following holds.
Let~$G$ be an \hbox{$r$-edge-coloured} graph on~$n\ge n_0$ vertices and with~$\delta(G)\ge(1-1/n_0)n$.
Then there exists a \hbox{$K_\ell$-tiling} in~$G$ consisting of monochromatic $K_\ell$ and the number of copies of~$K_\ell$ in the tiling is at least 
$$\left\lfloor\frac{(SR_r(\ell)-2)\delta(G)-(SR_r(\ell)-3)n}{\ell}\right\rfloor.$$
\end{theorem}

Observe that the bound on the number of copies of~$K_\ell$ in Theorem~\ref{further2} is optimal, as shown by the following construction.
Firstly, the case~$\ell=2$ is trivial since $SR_r (2)=2$ and an $n$-vertex graph can have at most $\lfloor n/2\rfloor$ vertex-disjoint copies of~$K_2$.
So fix~$r,\ell\in\mathbb N$ with~$r\ge2$ and~$\ell\ge3$, and let $n_0:=SR_r(\ell)-2\ge1$.
Pick any~$n,\delta\in\mathbb N$ with~$n-1\ge\delta>(1-1/n_0)n$. 
Consider a special $r$-edge-coloured copy $H$ of~$K_{SR_r(\ell)-1}$ which does not contain a monochromatic copy of~$K_\ell$; say the vertex~$v$ is not adjacent to any red edge.
Let~$G$ be the $r$-edge-coloured graph obtained by blowing up each vertex in~$V(H)\setminus\{v\}$ to a class of size~$n-\delta\ge1$ and~$v$ to a class~$U$ of size~$n-(n-\delta)(SR_r(\ell)-2)$.
Observe that
$$n-(n-\delta)(SR_r(\ell)-2)>n-(n-(1-1/n_0)n)(SR_r(\ell)-2)=0,$$
so the blow-up is well-defined.
Finally, add all edges inside~$U$ and colour them red.
Note that~$\delta(G)=n-(n-\delta)=\delta$.
Any monochromatic copy of~$K_\ell$ in~$G$ must have at least two vertices in~$U$, otherwise there would be a monochromatic~$K_\ell$ in~$H$, a contradiction.
However, since all edges in~$U$ are red and all edges between~$U$ and $V(G)\setminus U$ are not red, it follows that every monochromatic copy of~$K_\ell$ must lie completely in~$U$.
Therefore, if there are~$m$ vertex-disjoint monochromatic copies of~$K_\ell$ in~$G$ then
$$m\le\left\lfloor\frac{|U|}{\ell}\right\rfloor=\left\lfloor\frac{(SR_r(\ell)-2)\delta-(SR_r(\ell)-3)n}{\ell}\right\rfloor,$$
where the right-hand side of the above inequality matches the bound in Theorem~\ref{further2}.

Recall that the proof of case~\ref{case:Moonlarge} of Theorem~\ref{theorem:alaMoon} combined Lemma~\ref{lemma:K8}  with the Hajnal--Szemer\'edi theorem.
To prove Theorem~\ref{further2}, we need a generalisation of Lemma~\ref{lemma:K8} to larger cliques and multiple colours.
For  two colours, such a generalisation was  obtained by Burr, Erd\H os and Spencer~\cite{BurrES}, as mentioned in the introduction.
We obtain a further generalisation for more colours.

\begin{theorem}\label{further2-1}
For every $r,\ell\ge2$, there exists some  sufficiently large $t\in\mathbb N$ such that every $r$-edge-coloured $K_{t\ell+SR_r(\ell)-2}$ contains  $t$ vertex-disjoint monochromatic copies of~$K_\ell$.
\end{theorem}

Theorem~\ref{further2-1} is sharp.
Indeed, pick a special $r$-edge-colouring of $K_{SR_r(\ell)-1}$ that does not contain a monochromatic $K_{\ell}$; say  no edge of colour $c$ is adjacent to the vertex $v$ in $K_{SR_r(\ell)-1}$.
Then blow-up the vertex $v$ into a class $A$ of size $t\ell-1$, and add all edges inside $A$ and colour them  $c$.
The resulting graph is an $r$-edge-coloured $K_{t\ell+SR_r(\ell)-3}$.
Every monochromatic $K_\ell$ must lie in $A$, and thus there are no $t$  vertex-disjoint monochromatic $K_\ell$.

\begin{proof}[Proof of Theorem~\ref{further2-1}]
Take $t_0:=(\ell-1) R_r(\ell)$, $t:=R_r(t_0)$ and $n:=t\ell+SR_r(\ell)-2$.
It suffices to show that any $r$-edge-coloured copy $G$ of~$K_n$ contains~$t$ vertex-disjoint monochromatic copies of~$K_\ell$.
As~$n\ge t=R_r(t_0)$, there exists a monochromatic copy $K$ of $K_{t_0}$ in $G$, say a red copy.

Next, we construct a collection~$\mathcal H$ of vertex-disjoint monochromatic copies of~$K_\ell$.
This is achieved by repeatedly adding monochromatic copies of~$K_\ell$ to~$\mathcal H$ according to certain rules.
We write~$V(\mathcal H)$ to denote the set of vertices contained in some copy of $K_\ell$ in $\mathcal H$.
We also define $C:=V(K)\setminus V(\mathcal H)$ and $V:=V(G)\setminus (V(K)\cup V(\mathcal H))$.
Note that, as we add copies of~$K_\ell$ to~$\mathcal H$, the sets $V$, $C$ and~$V(\mathcal H)$ will be modified accordingly. 
It is convenient to think of~$V$ and~$C$ as the vertices in $V(G)\setminus V(K)$ and~$V(K)$ that have not yet been used to form copies of~$K_\ell$ in~$\mathcal H$.
Initially, we set $\mathcal H:=\emptyset$ and so we have $V=V(G)\setminus V(K)$ and $C=V(K)$.
We add copies of~$K_\ell$ to~$\mathcal H$ in three phases.

\medskip

{\bf Phase I.} If~$V$ contains a monochromatic copy of $K_\ell$, add it to~$\mathcal H$.
Iterate this as long as possible. 

\medskip

After Phase I is completed, we must have that there is no monochromatic copy of~$K_\ell$ in~$V$.
In particular, we have $|V|\le R_r(\ell)$.
Note that no copy of~$K_\ell$ added to~$\mathcal H$ so far intersects~$K$, hence we still have $C=V(K)$ and so $|C|=(\ell-1)R_r(\ell)\ge(\ell-1)|V|$.

\medskip

{\bf Phase II.} If there is a vertex $v\in V$ and a set $X\subseteq C$ with $|X|=\ell-1$ such that all edges between~$v$ and~$X$ are red, then $G[\{v\}\cup X]$ is a red copy of~$K_\ell$.
Add it to $\mathcal H$ and iterate as long as possible.

\medskip

After Phase II is completed, if $V=\emptyset$ then $V(\mathcal H)\cup V(K)=V(G)$. 
Since~$K$ is monochromatic, it is easy to see that there exist~$t$ vertex-disjoint monochromatic copies of~$K_\ell$ as required.
Hence, suppose that $V\neq\emptyset$.

Observe that after each iteration of Phase II, the quantities~$|V|$ and~$|C|$ decrease by~$1$ and~$\ell-1$, respectively.
Thus, after Phase II is completed, it is still the case that $|C|\ge(\ell-1)|V|>0$.

Since both $V$ and $C$ are non-empty, Phase II must have terminated because 
for every vertex $v\in V$ there are at most $\ell-2$ red edges between~$v$ and $C$.
This fact together with $|C|\ge(\ell-1)|V|$ implies there exists a subset $S\subseteq C$ of size $|S|=|V|$ such that none of the edges between $S$ and $V$ are coloured red.

\medskip

{\bf Phase III.} 
If $|V|\ge SR_r(\ell)-1$, then pick any vertex~$v\in S\setminus V(\mathcal H)$ and observe that  $G[V\cup\{v\}]$ has at least $SR_r(\ell)$ vertices and  is equipped with a special $r$-edge-colouring (namely, no edge incident to~$v$ is coloured red).
Therefore, $G[V\cup\{v\}]$ contains a monochromatic copy of $K_\ell$.
Add this copy to~$\mathcal H$, and iterate as long as it is possible.

\medskip

Observe that after each iteration of Phase III, the quantity $|V|$ decreases by either~$\ell-1$ or~$\ell$ while $|S\setminus V(\mathcal H)|$ decreases by either~$1$ or~$0$. 
In particular, after Phase~III has terminated we have $|V|\le|S\setminus V(\mathcal H)|$. 
Hence, it must be the case that Phase~III terminated because $|V|\le SR_r(\ell)-2$.

It follows that $|V(\mathcal H)\cup V(K)|\ge t\ell$.
Since $K$ is monochromatic, it is easy to see that $G[V(\mathcal H)\cup V(K)]$ contains~$t$ vertex-disjoint monochromatic copies of $K_\ell$, as required.
\end{proof}

We are now ready to prove Theorem~\ref{further2}.

\begin{proof}[Proof of Theorem~\ref{further2}]
Let $r, \ell\ge2$ and $t\in\mathbb N$ be sufficiently large so that  the statement of Theorem~\ref{further2-1} holds.
Let $n_0:=t\ell+SR_r(\ell)-2$.
It suffices to show that for any \hbox{$r$-edge-coloured} graph~$G$ on~$n\ge n_0$ vertices and with~$\delta(G)\ge(1-1/n_0)n$, there exists a \hbox{$K_\ell$-tiling} in~$G$ such that every copy of~$K_\ell$ is monochromatic and the number of copies of~$K_\ell$ in the tiling is at least 
$$\left\lfloor\frac{(SR_r(\ell)-2)\delta(G)-(SR_r(\ell)-3)n}{\ell}\right\rfloor.$$

We proceed by induction on~$n$.
Before this, we consider the case when $(1-1/n_0)n\le\delta(G)\le((n_0-1)n+\ell-1)/n_0$ (for all~$n\geq n_0$).
Note that any induced subgraph~$H$ of~$G$ with~$|H|=n_0(n-\delta(G))\le n$ satisfies~$\delta(H)\ge\delta(G)-(|G|-|H|)=(n_0-1)(n-\delta(G))=(1-1/n_0)|H|$.
Theorem~\ref{thm:HajnalSzm} implies~$H$, and thus~$G$, contains a~$K_{n_0}$-tiling consisting of~$n-\delta(G)$ copies of~$K_{n_0}$.
By Theorem~\ref{further2-1}, each~$K_{n_0}$ contains $t$ \hbox{vertex-disjoint} monochromatic copies of~$K_\ell$.
Taking the union of all such copies yields a~\hbox{$K_\ell$-tiling} consisting of precisely~$t(n-\delta(G))$ monochromatic copies of~$K_\ell$. 
This concludes the verification of this case, as
$$\left\lfloor\frac{(SR_r(\ell)-2)\delta(G)-(SR_r(\ell)-3)n}{\ell}\right\rfloor
=t(n-\delta(G))+\left\lfloor\frac{(SR_r(\ell)-2+t\ell)\delta(G)-(SR_r(\ell)-3+t\ell)n}{\ell}\right\rfloor$$
$$=t(n-\delta(G))+\left\lfloor\frac{n_0\delta(G)-(n_0-1)n}{\ell}\right\rfloor
\le t(n-\delta(G))+\left\lfloor\frac{\ell-1}{\ell}\right\rfloor=t(n-\delta(G)).$$

Now we can proceed by induction on~$n$.
The base cases  when $n_0\le n\le n_0+\ell-1$ are covered by the last paragraph.
Next, we check the inductive step.
Suppose $G$ is an $n$-vertex graph where~$n\ge n_0+\ell$.
By the previous paragraph we may assume that $\delta(G)\ge((n_0-1)n+\ell)/n_0$.

It is easy to show that~$G$ contains a~$K_{n_0}$ (e.g., by Theorem~\ref{thm:HajnalSzm2}), which in turn contains a monochromatic copy~$T$ of~$K_\ell$ by Theorem~\ref{further2-1}.
Let~$G':=G\setminus V(T)$.
Note that~$G'$ is an~\hbox{$r$-edge-coloured} graph on~$n-\ell$ vertices with minimum degree~$\delta(G')\ge\delta(G)-\ell\ge(1-1/n_0)(n-\ell)$.
By the inductive hypothesis,~$G'$ contains a~\hbox{$K_\ell$-tiling} consisting of
$$\left\lfloor\frac{(SR_r(\ell)-2)\delta(G')-(SR_r(\ell)-3)(n-\ell)}{\ell}\right\rfloor\ge\left\lfloor\frac{(SR_r(\ell)-2)\delta(G)-(SR_r(\ell)-3)n}{\ell}\right\rfloor-1$$
monochromatic copies of~$K_\ell$.
Adding~$T$ to this tiling yields a~\hbox{$K_\ell$-tiling} in~$G$ with the required number of monochromatic copies of~$K_\ell$.
This concludes the inductive step and the proof.
\end{proof}

\subsection{Open problems}
In general, 
Problems~\ref{problem:alaBES} and~\ref{problem:alaMoon} remain wide open. 
The main open problem which complements our work  is Question~\ref{conjecture:alaBES}; if true,  this would fully generalise Theorem~\ref{theorem:BES} to the minimum degree setting.
It would also be interesting to improve the error term in case~\ref{case:Moonmedium} of Theorem~\ref{theorem:alaMoon};
we believe that the $o(n)$ term should not appear at all. 
Ideas used in the proof of the Hajnal--Szemer\'edi theorem (see also~\cite{KiersteadK}) may be helpful.

Although our argument for case~\ref{case:BESlarge} of Theorem~\ref{theorem:alaBES} does not immediately generalise to larger cliques and multiple colours, it seems likely that some of the ideas used should be useful.
The main challenge is to prove analogues of Lemmas~\ref{lemma:bowtie1} and~\ref{lemma:bowtie2}.

Another potentially interesting future direction  is the asymmetric version of Problems~\ref{problem:alaBES} and~\ref{problem:alaMoon}; that is, the monochromatic tiling one seeks in each colour could be different now.
 
Finally, we remark that Ramsey-type results for tilings have been used in some interesting applications.
For example, in~\cite{GyarfasS2} an analogue of Theorem~\ref{theorem:BES} was proved where one insists that the monochromatic copy of $mK_3$ must lie in a connected subgraph of its own colour.
This result was used to determine the Ramsey number of an `almost' square of a cycle $C_m$. It would be interesting to see if our results in this paper have similar applications (perhaps via the regularity method).

\section*{Acknowledgments}
The work in this article was carried out through research visits supported by the BRIDGE strategic alliance between the University of Birmingham and the University of Illinois at Urbana-Champaign.
The authors are also grateful to the referees for their careful reviews.

\smallskip
{\noindent \bf Data availability statement.}
There are no additional data beyond that contained within the main manuscript.

\section*{Appendix}

In this appendix, we prove Lemma~\ref{theorem:embedding} using the regularity method. 
First, we introduce some  notation. 
The \emph{density} of a bipartite graph with vertex classes~$A$ and~$B$ is
defined to be 
$$d(A,B):=\frac{e(A,B)}{|A|\cdot|B|},$$
where here $e(A,B)$ is the number of edges between $A$ and $B$.
Given~$\eps>0$, a graph~$G$ and two disjoint sets~$A, B\subseteq V(G)$, we say that the pair~$(A, B)_G$  is \emph{$(\eps, d)$-regular} if~$d(A,B)\ge d$ and, for all sets~$X\subseteq A$ and~$Y\subseteq B$ with~$|X| \ge \eps|A|$ and~$|Y| \ge \eps|B|$, we have~$|d(A,B) - d(X,Y)| < \eps$.
The pair~$(A,B)_G$ is~$(\eps,d)$-{\it super-regular} if 
all sets~$X\subseteq A$ and~$Y\subseteq B$ with~$|X| \ge \eps|A|$ and~$|Y| \ge \eps|B|$ satisfy~$d(X,Y)\geq d$ and, furthermore, $d_G(a) \ge d|B|$ for all~$a \in A$ and~$d_G(b) \ge d|A|$ for all~$b \in B$.




The following  is the degree form of Szemer{\'e}di's Regularity Lemma.

\begin{lemma}[Regularity Lemma~\cite{Szemeredi}]\label{theorem:regularity}
For every~$\eps>0$ and every $\ell_0\in \mathbb N$ there is an~$M = M(\eps,\ell_0)$ such that for every $d\in[0,1)$ and for every graph~$G$ on~$n\ge M$ vertices, there exists a partition~$V_0,V_1,\dots,V_\ell$ of $V(G)$ and a spanning subgraph~$G'$ of~$G$ such that the following holds:
\begin{itemize}
    \item~$\ell_0 \le \ell \le M$ and~$|V_0| \le \eps n$,
\item~$|V_i| = |V_1|~$ for every~$i\in [\ell]$,
\item~$d_{G'}(x) \ge d_G(x) - (d+\eps)n$ for all~$x\in V(G)$,
\item for all~$i \in [\ell]$ the graph~$G'[V_i]$ is empty,
\item for all~$1 \le i < j \le \ell$, $(V_i, V_j )_{G'}$  either has density $0$ or is~$(\eps,d)$-regular.
\end{itemize}
\end{lemma}

The \emph{reduced graph~$R$ of~$G$ with parameters~$\eps$,~$\ell_0$ and~$d$} is the graph with vertex set~$\{V_i:i\in[\ell]\}$ and in which~$V_i V_j$ is an edge precisely when~$(V_i,V_j)_{G'}$ is~$(\eps, d)$-regular.
The following well-known consequence of the Regularity Lemma states that the reduced graph almost inherits the minimum degree of the original graph.

\begin{proposition}\label{reducedgraph}
Let $\ell_0 \in \mathbb N$, let~$0<\eps,d,k<1$ and let~$G$ be an~$n$-vertex graph with~$\delta(G)\geq kn$. If~$R$ is the reduced graph of~$G$ obtained by applying Lemma~\ref{theorem:regularity} with parameters~$\eps$,~$\ell_0$ and~$d$, then~$\delta(R)\geq(k-2\eps-d)|R|$.\qed
\end{proposition}



\begin{lemma}[Blow-up Lemma~\cite{blowup}]\label{blowuplemma} 
Given a graph~$R$ of order~$\ell$ and~$d,\Delta>0$, there exists an~$\eps > 0$ such that the following holds. Given any $m \in \mathbb N$, let 
$V_1, V_2,\dots, V_\ell$ denote the vertex classes of the blow-up $R(m)$ of $R$ (so $|V_i|=m$ for all $i \in [\ell])$.
Let $G$ be a graph obtained from $R(m)$ as follows: for every $1\leq i<j \leq \ell$ such that $(V_i,V_j)$ induces a complete bipartite graph in $R(m)$, $(V_i,V_j)_G$ now forms an $(\eps,d )$-super-regular pair.
 If a graph~$H$ with~$\Delta(H) \le \Delta$ lies in~$R(m)$, then there is a copy of $H$ in~$G$.
\end{lemma}

Lemma~\ref{theorem:embedding} now follows easily from Lemmas~\ref{theorem:regularity} and~\ref{blowuplemma}.
In the proof below,  constants in the displayed hierarchy are chosen from right to left.


\subsection*{Proof of Lemma~\ref{theorem:embedding}}
Given $\eta>0$, choose constants
$$0<1/n_0 \ll 1/\ell_0 \ll \eps \ll d\ll\eta, 1/8.$$
Given $n \geq n_0$, let $G$ be an $n$-vertex graph as in the statement of the lemma.
Apply Lemma~\ref{theorem:regularity} to the graph~$G$ with parameters~$\eps$,~$\ell_0$ and~$d$ to obtain $\ell\in \mathbb N$, a partition~$V_0,V_1,\dots,V_\ell$ of~$G$, a spanning subgraph $G'$ of $G$,  and a reduced graph~$R$ of $G$. Set $m:=|V_1|$.

By Proposition~\ref{reducedgraph} we  have~$\delta(R)\ge(\delta(G)/n-2\eps-d)|R|\ge(\delta(G)/n-\eta/4)|R|$, and so~$\delta(R)/|R|\ge\delta(G)/n-\eta/4$.
By greedily deleting edges, we may further assume that~$\delta(R)/|R|\le\delta(G)/n$ and so property~\ref{property:i} holds.
Property~\ref{property:ii} also holds as~$\eps\le\eta/2$.
It remains to verify property~\ref{property:iii}.

Let~$V_{i_1},\dots,V_{i_r}$ form a clique in~$R$ with~$r\le 8$. 
By deleting `small degree' vertices, for each $j \in [r]$, one obtains a
 set~$V_{i_j}'\subseteq V_{i_j}$ such that $|V_{i_j}'|=\lceil (1-r\eps)m \rceil$
 and so that 
 $(V_{i_j}',V_{i_k}')_{G'}$ is~$(2\eps,d-r\eps)$-super-regular for each distinct $j,k \in [r]$.  
Now by Lemma~\ref{blowuplemma},
$G'[V'_{i_1},\dots,V'_{i_r}]$ (and thus
$G[V_{i_1},\dots,V_{i_r}]$) contains a $K_r(2)$-tiling consisting  of at least~$(1-\eta/2)m/2$ copies of~$K_r(2)$; so indeed~\ref{property:iii} holds. \qed

\end{document}